\documentclass[hidelinks,onefignum,onetabnum]{siamart251216}

\usepackage{color}
\usepackage{booktabs,graphicx}
\usepackage{subcaption}
\usepackage{amssymb}
\usepackage{tikz}
\usepackage{multirow}
\usepackage{siunitx}

\newsiamthm{assumption}{Assumption}
\newcommand{\jump}[1]{[\![ #1 ]\!]}
\newcommand{\avg}[1]{\{ #1 \}}

\usepackage{lipsum}
\usepackage{amsfonts}
\usepackage{graphicx}
\usepackage{epstopdf}
\usepackage{algorithmic}
\ifpdf
  \DeclareGraphicsExtensions{.eps,.pdf,.png,.jpg}
\else
  \DeclareGraphicsExtensions{.eps}
\fi

\newsiamremark{remark}{Remark}
\newsiamremark{hypothesis}{Hypothesis}
\crefname{hypothesis}{Hypothesis}{Hypotheses}
\newsiamthm{claim}{Claim}
\newsiamremark{fact}{Fact}
\crefname{fact}{Fact}{Facts}

\headers{Trace inequality \& Kernel-based Meshfree Galerkin}{L. Ling, S. Ma, W. Qiu, and Y. Sun}

\title{Trace Inverse Estimates and Kernel-based Meshfree Galerkin Methods for Elliptic Interface Problems\thanks{Submitted to .
\funding{This work was supported by the General Research Fund (GRF No. 12301824, 12301526) of the Hong Kong Research Grant Council, the Guangdong and Hong Kong Universities “1+1+1” Joint Research Collaboration Scheme (project no. 2025A0505000014), the National Natural Science Foundation of China (Young Scientists Fund, Category C, Grant No. 12601787) and the Guangdong Basic and Applied Basic Research Foundation (Grant No. 2026A1515012472).}}}

\author{Leevan Ling\thanks{Department of Mathematics, Hong Kong Baptist University, Hong Kong
  (\email{lling@hkbu.edu.hk}, \email{shu\_ma@hkbu.edu.hk}, \email{yzsun95@hkbu.edu.hk}).}
\and Shu Ma\footnotemark[2]
\and Weifeng Qiu\thanks{Department of Mathematics, City University of Hong Kong, Hong Kong
(\email{weifeqiu@cityu.edu.hk}).}
\and Yizhong Sun\footnotemark[2]}

\usepackage{amsopn}

\ifpdf
\hypersetup{
  pdftitle={Trace inequality \& Kernel-based Meshfree Galerkin},
  pdfauthor={L. Ling, S. Ma, WF. Qiu, and YZ. Sun}
}
\fi

\begin{document}

\maketitle
\raggedbottom

\begin{abstract}
Kernel Galerkin methods allow meshfree trial spaces to be
generated independently in each material subdomain and
refined through their center sets, without generating an approximation mesh.
We establish trace inverse estimates for products of
Sobolev-kernel spaces on Lipschitz subdomains and use them
to analyze symmetric Nitsche discretizations of linear elliptic interface problems.
The estimates control gradient traces on subdomain boundaries
by subdomain gradients, the interface jump, and an exterior
boundary trace.
Their conormal form supplies the flux control required for
stability directly from the kernel and center assumptions,
without a separate discrete flux hypothesis.
For quasi-uniform center families with comparable fill distances
and a fixed kernel scale, we obtain well-posedness for sufficiently
large penalty parameters independent of refinement.
We prove energy and $L^2$ error bounds of orders $h^{s-1}$
and $h^s$ for piecewise $H^s$ solutions, $2\le s\le m$,
where $h$ is the maximum subdomain fill distance and $m$
is the kernel Sobolev order.
The $L^2$ bound additionally requires piecewise $H^2$ regularity
of the homogeneous adjoint interface problem.
For kernels $\Phi_m(\varepsilon\,\cdot)$ with $0<\varepsilon\le1$,
we derive a scale-explicit condition-number bound in the
kernel-translate basis and an $O(h^{-2})$ bound in an
$L^2$-orthonormal basis.
The analysis assumes exact geometry and integration.
Planar experiments examine convergence and kernel scaling,
while a three-dimensional heat-conduction problem in a curved
composite illustrates the method.
\end{abstract}



\section{Introduction}
\label{sec:introduction}
Second-order linear elliptic interface problems arise in heterogeneous media models such as electrochemical batteries, oxygen transport, heat conduction, and electrostatics with discontinuous dielectric properties \cite{TaleeiDehghan2014, MartinFornberg2017,  Pastor-Alonso+BergETAL-Modeoxygtranbrai:24, Sitaraman+AllenETAL-immeintemethmicr:26}. We consider scalar problems in which prescribed jumps of the solution and conormal flux couple the equations in adjacent materials across possibly curved interfaces and
exterior boundaries. The model problem is specified in Section~\ref{subsec:interface-setting}.

Finite element methods provide fitted and unfitted treatments
of interface problems with established convergence theory
\cite{HansboHansbo2002,LehrenfeldReusken2018}.
Fitted discretizations align element boundaries with an interface
approximation, as shown in Fig.~\ref{fig:intro-discretizations}(a).
Their generation and refinement must maintain geometric accuracy
and mesh quality, particularly for curved high-order elements
\cite{LuShephardTendulkarBeall2014}.
Unfitted methods relax this alignment requirement by allowing
the interface to cut a background mesh, as illustrated in
Fig.~\ref{fig:intro-discretizations}(b).
High-order accuracy then requires suitable geometric approximation
and integration over cut cells \cite{LehrenfeldReusken2018}.
Stabilization techniques, including ghost penalties, control
the effects of small cut portions
\cite{BurmanClausHansboLarsonMassing2015}.

The geometric and element-quality requirements of mesh
construction motivate a trial space whose resolution can
be prescribed directly through points in each material,
without generating or modifying a fitted approximation mesh to define
and refine the trial space.
Radial basis function (RBF) approximation provides such a
meshfree construction through linear combinations of functions
centered at scattered points, each depending only on the
distance from its center
\cite{Wendland1999,SchabackWendland2006}.
Fig.~\ref{fig:intro-discretizations}(c) illustrates centers
within the physical subdomains, with none required on the
interface or exterior boundary.
We use Sobolev kernels, including Mat\'ern and Wendland RBFs,
whose Fourier decay identifies their native spaces with
$H^m$ up to equivalent norms
in Section~\ref{subsec:kernel-setting}.
For a single-domain Dirichlet problem, Chernih and Le Gia
\cite{ChernihLeGia2014} analyze a Nitsche kernel Galerkin method under
an additional hypothesis controlling the boundary normal derivative by
the volume gradient.  Their analysis therefore does not derive this
discrete flux control from the kernel and center assumptions themselves.
\begin{figure}[tbp]
\centering
\begin{subfigure}[t]{0.325\textwidth}
    \centering
    \includegraphics[width=\textwidth]{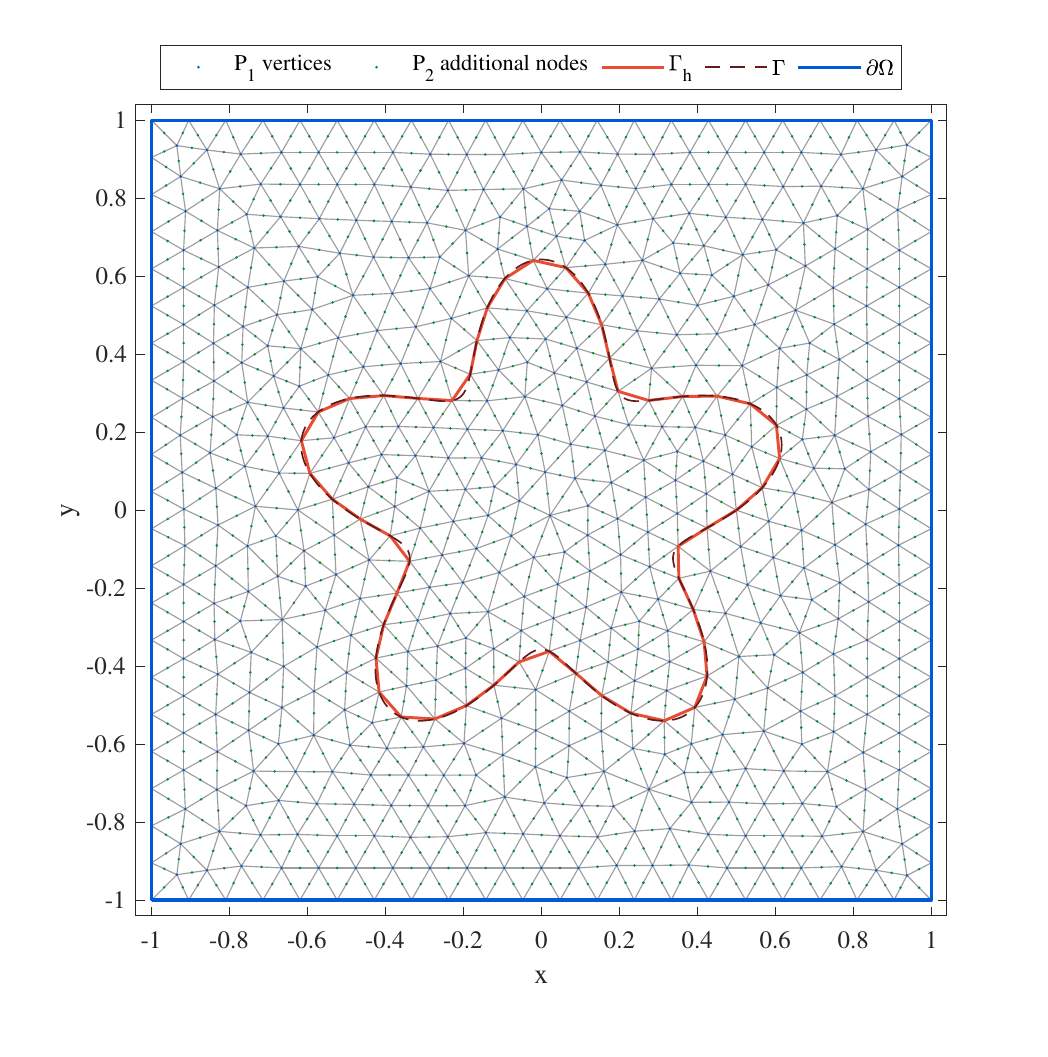}
    \caption{Fitted finite elements.}
\end{subfigure}
\hfill
\begin{subfigure}[t]{0.325\textwidth}
    \centering
    \includegraphics[width=\textwidth]{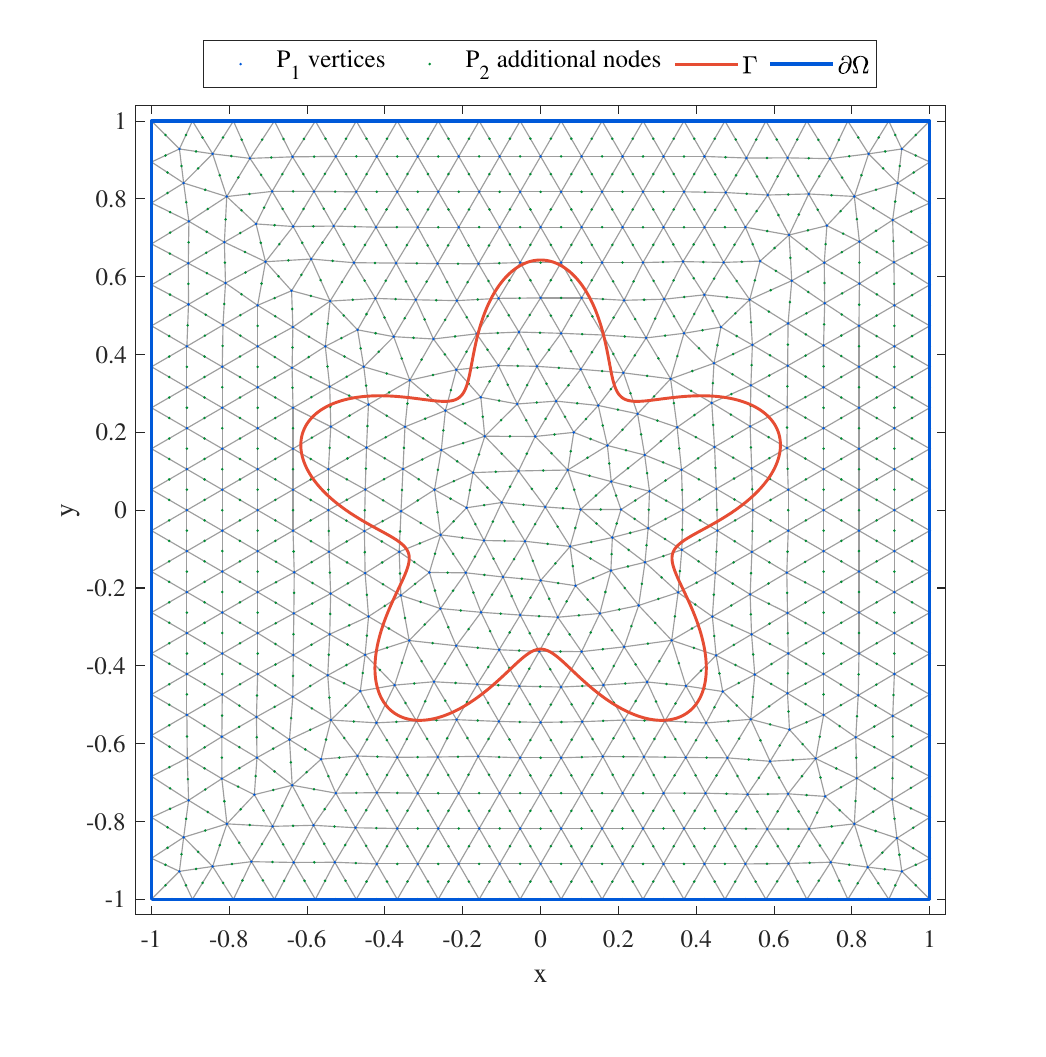}
    \caption{Unfitted finite elements.}
\end{subfigure}
\hfill
\begin{subfigure}[t]{0.325\textwidth}
    \centering
    \includegraphics[width=\textwidth]{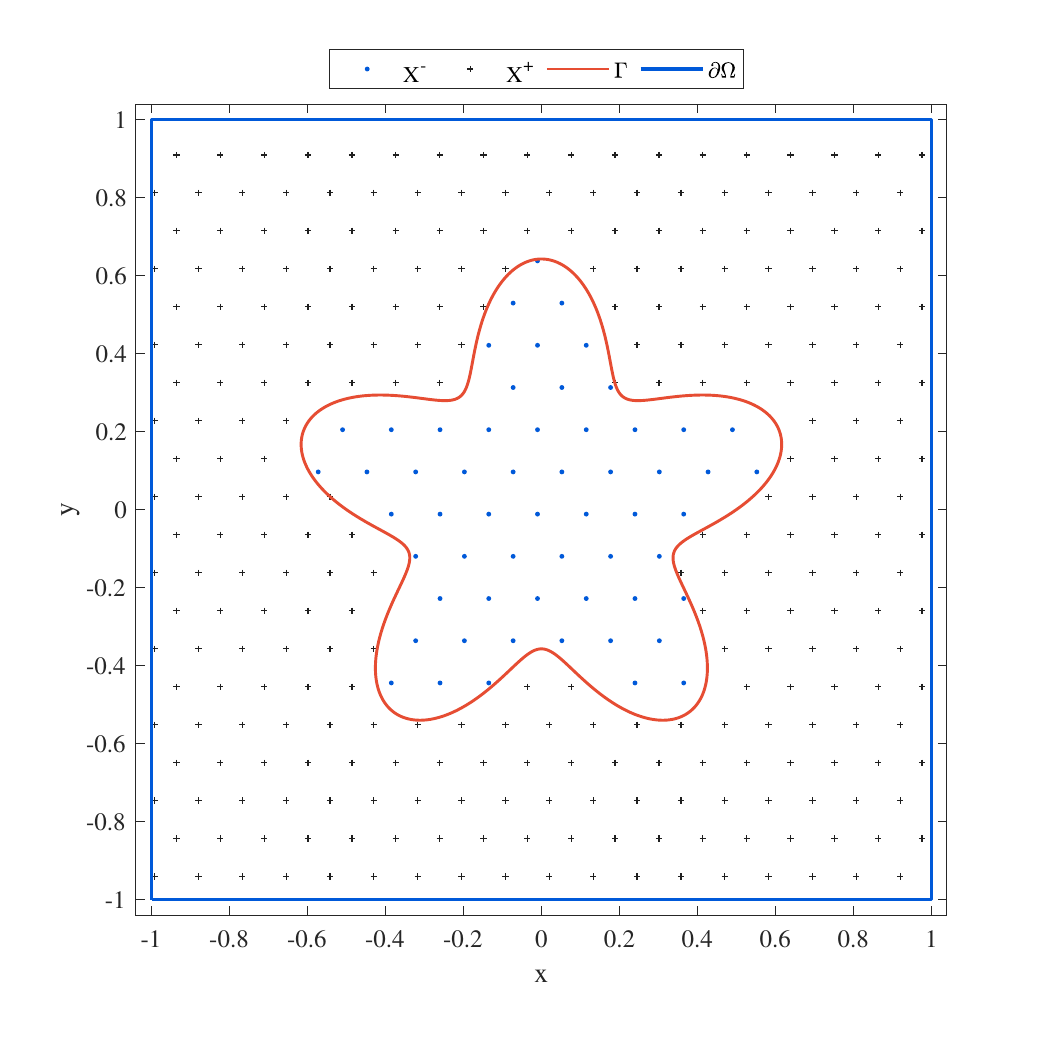}
    \caption{Kernel trial centers.}
\end{subfigure}
\vspace{-4mm}
\caption{
Trial-space constructions for a curved interface.
The fitted mesh follows the interface, the unfitted background mesh
is cut by it, and the kernel centers are generated independently
inside the two physical subdomains.
}
\label{fig:intro-discretizations}
\vspace{-4mm}
\end{figure}

For the interface problem, we use two independently generated kernel
spaces, one on each physical subdomain, and couple them through a
symmetric Nitsche formulation~\eqref{eq:discrete-nitsche}.  Neither center set is constrained
to the interface or exterior boundary, while the exterior Dirichlet
condition and the prescribed solution jump are imposed weakly.
The resulting method differs from the single-domain analysis of
Chernih and Le Gia \cite{ChernihLeGia2014} in two essential respects:
it couples independent kernel spaces across an interface, and the
boundary/interface flux control required for Nitsche stability is
derived from the Sobolev-kernel and center assumptions rather than
postulated as an additional discrete hypothesis.
Galerkin testing then gives the kernel coefficient system
\eqref{eq:kernel-galerkin-system}, while the volume and surface
quadrature can be constructed independently of the trial centers.

Deriving this flux control is the main analytical difficulty.
In finite element analysis, suitable discrete trace inequalities are a
standard mechanism for controlling Nitsche flux terms.  For example,
Guo and Lin \cite{GuoLin2019Cauchy} establish trace inequalities for
high-order immersed finite element spaces constructed through local
Cauchy extensions.
The kernel spaces considered here have a different structure.
Kernel translates are not elementwise polynomials, and their derivatives
need not belong to the original trial space.  Applying the continuous
trace theorem to a gradient therefore introduces an $H^2$ norm of the
trial function.  We control this higher-order norm  the full-range Bernstein inequality 
\cite{SunLing2025}, and combine with the
interface Poincar\'e inequality of Lemma~\ref{lem:interface-poincare}.
This yields the product-space trace inverse estimate of
Theorem~\ref{thm:trace-inverse}, which bounds the gradient traces on the
subdomain boundaries by the volume gradients, the interface jump, and
one exterior boundary trace, with a constant independent of refinement.

With this discrete flux control in hand, the remaining variational
analysis follows without an additional kernel-specific stability
hypothesis.  Under exact geometry and integration,
Theorem~\ref{thm:well-posedness} proves well-posedness for sufficiently
large penalty parameters chosen independently of refinement.
If the exact solution has piecewise regularity
$u^\eta\in H^s(\Omega^\eta)$ with  $\eta=\pm$ and $2\le s\le m$,
then the energy and $L^2$ estimates of Theorems~\ref{thm:energy-error} and~\ref{thm:L2-error} reduce, in particular, to
\begin{equation*}
\|\nabla(\mathbf u-\mathbf u_h)\|_{\mathbf L^2(\boldsymbol\Omega)}
\le
C h^{s-1}\|\mathbf u\|_{\mathbf H^s(\boldsymbol\Omega)},
\qquad
\|\mathbf u-\mathbf u_h\|_{\mathbf L^2(\boldsymbol\Omega)}
\le
C h^s\|\mathbf u\|_{\mathbf H^s(\boldsymbol\Omega)},
\end{equation*}
where the second estimate additionally assumes piecewise $H^2$
regularity of the homogeneous adjoint interface problem.
Section~\ref{subsec:conditioning} further quantifies the dependence
of coefficient conditioning on the kernel scale and shows the
$O(h^{-2})$ operator-conditioning behavior after $L^2$ mass scaling.
The numerical experiments in Section~\ref{sec:numerics} separately
examine convergence, kernel-scale effects on accuracy and conditioning,
and a three-dimensional heat-conduction problem on a curved composite.

\section{Interface problem and symmetric Nitsche formulation}
\label{sec:setting}

We specify the interface problem and its regularity assumptions, then
introduce a symmetric Nitsche formulation on general product trial
spaces. The stability analysis leads to a discrete trace estimate for
the trial spaces. Throughout this section, $C>0$ denotes a
generic constant independent of the discretization scale.

\subsection{Model problem and regularity assumptions}
\label{subsec:interface-setting}

Let $\Omega\subset\mathbb R^d$, $d\ge2$, be a bounded connected
Lipschitz domain partitioned into connected Lipschitz subdomains
$\Omega^\pm$ with disjoint interiors and
$\Omega=\operatorname{int}
(\overline{\Omega^-}\cup\overline{\Omega^+})$.
Define the interface and exterior boundary portions by
$\Gamma:=(\partial\Omega^-\cap\partial\Omega^+)\cap\Omega$
and $\Gamma_{\Omega^\eta}:=
\partial\Omega^\eta\cap\partial\Omega$ for $\eta=\pm$.
We assume $|\Gamma|_{d-1}>0$ and label the subdomains so that
$|\Gamma_{\Omega^+}|_{d-1}>0$.
Terms over an empty exterior boundary portion are omitted.
Both inclusion and separating geometries are allowed, as illustrated
in Fig.~\ref{fig:interface-geometries}.
The unit normal $\vec n_\Gamma$ points from $\Omega^-$ to $\Omega^+$.
The unit normal $\vec n_{\partial\Omega}$ points outward from $\Omega$.

For scalar-, vector-, or matrix-valued subdomain quantities $q^\pm$,
we use the product notation $\mathbf q:=(q^-,q^+)$.
Products and differential operators act componentwise.
The corresponding ordered pairs of domains and exterior boundary
portions are $\boldsymbol\Omega:=(\Omega^-,\Omega^+)$ and
$\Gamma_{\boldsymbol\Omega}:=
(\Gamma_{\Omega^-},\Gamma_{\Omega^+})$.
For product functions with traces on $\Gamma$, define the jump
$\jump{\mathbf v}:=v^--v^+$ and arithmetic average
$\avg{\mathbf v}:=\tfrac12(v^-+v^+)$.
The interface problem can then be written in the equivalent forms
\begin{equation}
\left\{
\begin{alignedat}{2}
-\nabla\!\cdot(A^\pm\nabla u^\pm)&=f^\pm
&\quad&\text{in }\Omega^\pm,\\
u^\pm&=g_D^\pm&&\text{on }\Gamma_{\Omega^\pm},\\
u^--u^+&=g_\Gamma&&\text{on }\Gamma,\\
(A^-\nabla u^--A^+\nabla u^+)\cdot\vec n_\Gamma
&=\psi_\Gamma&&\text{on }\Gamma,
\end{alignedat}
\right.
\Leftrightarrow
\left\{
\begin{alignedat}{2}
-\nabla\!\cdot(\mathbf A\nabla\mathbf u)&=\mathbf f
&\quad&\text{in }\boldsymbol\Omega,\\
\mathbf u&=\mathbf g_D&&\text{on }\Gamma_{\boldsymbol\Omega},\\
\jump{\mathbf u}&=g_\Gamma&&\text{on }\Gamma,\\
\jump{\mathbf A\nabla\mathbf u}\cdot\vec n_\Gamma
&=\psi_\Gamma&&\text{on }\Gamma.
\end{alignedat}
\right.
\label{eq:intro-interface-problem}
\end{equation}
Here $f^\eta$ is the volume source, and $g_\Gamma$ and
$\psi_\Gamma$ prescribe the jumps of the solution and conormal flux.
The product boundary condition is understood componentwise.

For $\eta=\pm$, assume that the diffusion tensor
$A^\eta\in[W^{1,\infty}(\Omega^\eta)]^{d\times d}$
is symmetric and satisfies the uniform ellipticity bounds with constants $0<\alpha_0\le\alpha_1<\infty$, 
\begin{equation}
\alpha_0|\vec\xi|^2
\le
\vec \xi^\top A^\eta(\vec x)\vec\xi
\le
\alpha_1|\vec \xi|^2,
\qquad
\vec\xi\in\mathbb R^d,
\quad\text{a.e. }\vec x\in\Omega^\eta.
\label{eq:uniform-ellipticity}
\end{equation}

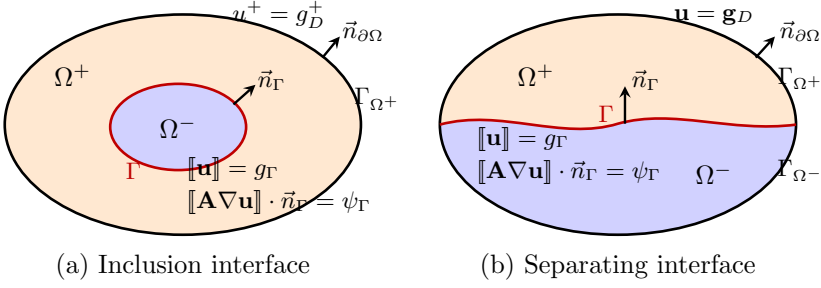
\begin{figure}[tbp]
\centering
\begin{tikzpicture}[
    scale=0.62,
    outerboundary/.style={black, line width=1pt},
    interface/.style={red!75!black, line width=1pt},
    normalarrow/.style={->,>=stealth,line width=0.9pt},
    domainlabel/.style={font=\normalsize},
    textlabel/.style={font=\small}
]

\begin{scope}[xshift=0cm]

\fill[orange!18] (0,0) ellipse (3.8 and 2.35);
\fill[blue!18]   (-0.10,-0.05) ellipse (1.45 and 0.92);

\draw[outerboundary] (0,0) ellipse (3.8 and 2.35);
\draw[interface]     (-0.10,-0.05) ellipse (1.45 and 0.92);

\node[domainlabel] at (-0.10,-0.05) {$\Omega^{-}$};
\node[domainlabel] at (-2.35,1.05) {$\Omega^{+}$};

\node[textlabel, red!75!black] at (-1.05,-0.98) {$\Gamma$};
\node[textlabel] at (4.15,0.55) {$\Gamma_{\Omega^+}$};

\draw[normalarrow] (1.08,0.43) -- (1.55,0.85);
\node[textlabel] at (1.90,0.95) {$\vec n_{\Gamma}$};

\draw[normalarrow] (2.98,1.43) -- (3.35,1.88);
\node[textlabel] at (3.85,2.00) {$\vec n_{\partial\Omega}$};

\node[textlabel, align=left] at (2.05,-1.30)
{
$\jump{\mathbf u}=g_{\Gamma}$\\[0.5mm]
$\jump{\mathbf A\nabla\mathbf u}\cdot\vec n_{\Gamma}
=\psi_{\Gamma}$
};

\node[textlabel] at (2.05,2.35) {$ u^+=g_D^+$};

\node[domainlabel] at (0,-3.02) {(a) Inclusion interface};

\end{scope}

\hspace{3mm}

\begin{scope}[xshift=8.8cm]

\fill[orange!18] (0,0) ellipse (3.8 and 2.35);

\begin{scope}
\clip (0,0) ellipse (3.8 and 2.35);
\fill[blue!18]
(-3.8,-2.6) --
(-3.8,0.0)
.. controls (-2.5,0.35) and (-1.2,-0.35) .. (0.0,0.02)
.. controls (1.2,0.35) and (2.4,-0.18) .. (3.8,0.0)
-- (3.8,-2.6) -- cycle;
\end{scope}

\draw[outerboundary] (0,0) ellipse (3.8 and 2.35);

\draw[interface]
(-3.8,0.0)
.. controls (-2.5,0.35) and (-1.2,-0.35) .. (0.0,0.02)
.. controls (1.2,0.35) and (2.4,-0.18) .. (3.8,0.0);

\node[domainlabel] at (-1.75,1.05) {$\Omega^{+}$};
\node[domainlabel] at (2.05,-1.05) {$\Omega^{-}$};

\node[textlabel, red!75!black] at (-0.25,0.22) {$\Gamma$};
\node[textlabel] at (3.90,0.98) {$\Gamma_{\Omega^+}$};
\node[textlabel] at (3.90,-0.98) {$\Gamma_{\Omega^-}$};

\draw[normalarrow] (0.15,0.02) -- (0.15,0.82);
\node[textlabel] at (0.60,0.95) {$\vec n_{\Gamma}$};

\draw[normalarrow] (2.94,1.43) -- (3.35,1.88);
\node[textlabel] at (3.90,2.00) {$\vec n_{\partial\Omega}$};

\node[textlabel, align=left] at (-1.05,-0.65)
{
$\jump{\mathbf u}=g_{\Gamma}$\\[0.5mm]
$\jump{\mathbf A\nabla\mathbf u}\cdot\vec n_{\Gamma}
=\psi_{\Gamma}$
};

\node[textlabel] at (2.05,2.30) {$\mathbf u=\mathbf g_D$};

\node[domainlabel] at (0,-3.02) {(b) Separating interface};

\end{scope}

\end{tikzpicture}
\vspace{-2mm}
\caption{Two representative interface geometries.
In the inclusion geometry, $\Gamma_{\Omega^-}=\varnothing$, whereas in the
separating geometry both $\Gamma_{\Omega^-}$ and $\Gamma_{\Omega^+}$ may be nonempty.}
\label{fig:interface-geometries}
\vspace{-4mm}
\end{figure}

For $t\ge0$, define the product Sobolev spaces
\begin{equation}
\mathbf H^t(\boldsymbol\Omega)
:=H^t(\Omega^-)\times H^t(\Omega^+),
\qquad
\mathbf L^2(\boldsymbol\Omega):=\mathbf H^0(\boldsymbol\Omega),
\label{eq:broken-space}
\end{equation}
with norms and seminorms
\begin{equation}
\|\mathbf v\|_{\mathbf H^t(\boldsymbol\Omega)}^2
:=\sum_{\eta=\pm}\|v^\eta\|_{H^t(\Omega^\eta)}^2,
\qquad
|\mathbf v|_{\mathbf H^t(\boldsymbol\Omega)}^2
:=\sum_{\eta=\pm}|v^\eta|_{H^t(\Omega^\eta)}^2.
\label{eq:broken-norm}
\end{equation}
Volume and surface $L^2$ inner products are denoted by
$(\cdot,\cdot)_{L^2(D)}$ and $\langle\cdot,\cdot\rangle_{L^2(\Sigma)}$.
For product functions, set
$(\mathbf p,\mathbf q)_{\mathbf L^2(\boldsymbol\Omega)}
:=\sum_{\eta=\pm}(p^\eta,q^\eta)_{L^2(\Omega^\eta)}$ with corresponding norm
$\|\mathbf v\|_{\mathbf L^2(\boldsymbol\Omega)}^2
:=(\mathbf v,\mathbf v)_{\mathbf L^2(\boldsymbol\Omega)}$. On the exterior boundary portions, define
$\langle\mathbf p,\mathbf q\rangle_{\mathbf L^2(\Gamma_{\boldsymbol\Omega})}
:=
\sum_{\eta=\pm}
\langle p^\eta,q^\eta\rangle_{L^2(\Gamma_{\Omega^\eta})}$ with
$\|\mathbf v\|_{\mathbf L^2(\Gamma_{\boldsymbol\Omega})}^2
:=
\langle\mathbf v,\mathbf v\rangle_{\mathbf L^2(\Gamma_{\boldsymbol\Omega})}$.
For the Dirichlet data, set
$\mathbf H^{1/2}(\Gamma_{\boldsymbol\Omega})
:=
H^{1/2}(\Gamma_{\Omega^-})
\times H^{1/2}(\Gamma_{\Omega^+})$.
Sobolev spaces on boundary portions are understood as restriction
spaces.
The data satisfy
\[
\mathbf f\in\mathbf L^2(\boldsymbol\Omega),\qquad
\mathbf g_D\in\mathbf H^{1/2}(\Gamma_{\boldsymbol\Omega}),\qquad
g_\Gamma\in H^{1/2}(\Gamma),\qquad
\psi_\Gamma\in L^2(\Gamma).
\]

We assume that the prescribed value data admit a lifting
$\mathbf g\in\mathbf H^1(\boldsymbol\Omega)$ satisfying
$\mathbf g=\mathbf g_D$ on $\Gamma_{\boldsymbol\Omega}$ and
$\jump{\mathbf g}=g_\Gamma$ on $\Gamma$.
Strong imposition of these conditions requires the correction
$\mathbf u-\mathbf g$ to have zero exterior boundary trace and
matching interface traces.
The Sobolev space of homogeneous corrections is
\[\mathbf H_0^1(\boldsymbol\Omega)
:=
~\left\{
\mathbf v\in\mathbf H^1(\boldsymbol\Omega):
\mathbf v=\mathbf0~\text{ on }~\Gamma_{\boldsymbol\Omega}, \quad
\jump{\mathbf v}=0~\text{ on }~\Gamma
\right\}.\]
Restriction to the subdomains identifies $H_0^1(\Omega)$ with
$\mathbf H_0^1(\boldsymbol\Omega)$.
The continuous weak problem is to find
$\mathbf u\in\mathbf g+\mathbf H_0^1(\boldsymbol\Omega)$ such that
\begin{equation}
(\mathbf A\nabla\mathbf u,\nabla\mathbf v)_{\mathbf L^2(\boldsymbol\Omega)}
=
(\mathbf f,\mathbf v)_{\mathbf L^2(\boldsymbol\Omega)}
+\langle\psi_\Gamma,\avg{\mathbf v}\rangle_{L^2(\Gamma)}
\qquad
\forall\,\mathbf v\in\mathbf H_0^1(\boldsymbol\Omega).
\label{eq:continuous-weak}
\end{equation}
Uniform ellipticity and the Poincar\'e inequality give a unique
weak solution.

For the error analysis, we assume
$\mathbf u\in\mathbf H^s(\boldsymbol\Omega)$ with $s\ge2$,
where $s$ denotes the available piecewise solution regularity.
The $L^2$ error estimate in Section~\ref{sec:analysis} requires
an additional regularity assumption.
For $\mathbf r\in\mathbf L^2(\boldsymbol\Omega)$, let
$\mathbf z\in\mathbf H_0^1(\boldsymbol\Omega)$ be the weak
solution of the homogeneous adjoint interface problem
\begin{equation}
\left\{
\begin{alignedat}{2}
-\nabla\!\cdot(\mathbf A\nabla\mathbf z)
&=\mathbf r
\quad\text{in }\boldsymbol\Omega,
&\qquad \mathbf z
&=\mathbf0
\quad\text{on }\Gamma_{\boldsymbol\Omega},\\
\jump{\mathbf z}
&=0
\quad\text{on }\Gamma,
&\qquad
\jump{\mathbf A\nabla\mathbf z}\cdot\vec n_\Gamma
&=0
\quad\text{on }\Gamma.
\end{alignedat}
\right.
\label{eq:adjoint-interface-problem}
\end{equation}
The adjoint has the same differential operator because $A^\eta$
is symmetric.
We assume
\begin{equation}
\mathbf z\in\mathbf H^2(\boldsymbol\Omega),
\qquad
\|\mathbf z\|_{\mathbf H^2(\boldsymbol\Omega)}
\le C_{\rm reg}\|\mathbf r\|_{\mathbf L^2(\boldsymbol\Omega)}
\qquad
\forall\,\mathbf r\in\mathbf L^2(\boldsymbol\Omega),
\label{eq:adjoint-regularity}
\end{equation}
with constant $C_{\rm reg}$ independent of $\mathbf r$.
For example, this estimate holds for positive piecewise constant
scalar diffusion on a convex polygon in $\mathbb R^2$ with a closed
$C^2$ interface contained strictly inside the domain
\cite[Theorem~2.1]{ChenZou1998}.
It remains an additional assumption for the Lipschitz geometries
considered here and is used only in the $L^2$ error analysis.

\subsection{Symmetric Nitsche formulation on product trial spaces}
\label{subsec:nitsche-formulation}
Let
$\mathbf V_h:=V_h^-\times V_h^+
\subset\mathbf H^1(\boldsymbol\Omega)$
be a finite-dimensional product space with well-defined one-sided
conormal traces in $L^2$ on the interface and exterior boundary. No boundary or interface value constraints are imposed on $\mathbf V_h$.
Examples include conforming finite element spaces on separate
subdomain meshes \cite{BeckerHansboStenberg2003} and the kernel
trial spaces constructed in Section~\ref{sec:trace-inverse}.
For finite element functions, the conormal traces are taken
elementwise.
No boundary or interface value constraints are imposed on
$\mathbf V_h$.

Let $0<h\le1$ denote the discretization scale, and let
$\gamma_\Gamma>0$ and $\gamma_{\partial,\eta}>0$, $\eta=\pm$,
be penalty parameters independent of $h$. Write
$\boldsymbol\gamma_\partial:=
(\gamma_{\partial,-},\gamma_{\partial,+})$.
Following the symmetric Nitsche method
\cite{Nitsche1971,BeckerHansboStenberg2003}, we seek
$\mathbf u_h\in\mathbf V_h$ such that
\begin{equation}
a_h(\mathbf u_h,\mathbf v_h)=\ell_h(\mathbf v_h)
\qquad
\forall\,\mathbf v_h\in\mathbf V_h,
\label{eq:discrete-nitsche}
\end{equation}
where the Nitsche bilinear form is
\begin{align}
a_h(\mathbf w,\mathbf v)
:={}&
(\mathbf A\nabla\mathbf w,\nabla\mathbf v)_{\mathbf L^2(\boldsymbol\Omega)} \nonumber
\\[2pt]
&\underbrace{
-\langle
\avg{\mathbf A\nabla\mathbf w}\cdot\vec n_\Gamma,
\jump{\mathbf v}
\rangle_{L^2(\Gamma)}
-\langle
\jump{\mathbf w},
\avg{\mathbf A\nabla\mathbf v}\cdot\vec n_\Gamma
\rangle_{L^2(\Gamma)}
}_{\text{interface flux terms}}
+\underbrace{
\frac{\gamma_\Gamma}{h}
\langle\jump{\mathbf w},\jump{\mathbf v}\rangle_{L^2(\Gamma)}
}_{\text{interface penalty term}} \label{eq:nitsche-bilinear}
\\[2pt]
&\underbrace{
-\langle
\mathbf A\nabla\mathbf w\cdot\vec n_{\partial\Omega},
\mathbf v
\rangle_{\mathbf L^2(\Gamma_{\boldsymbol\Omega})}
-\langle
\mathbf w,
\mathbf A\nabla\mathbf v\cdot\vec n_{\partial\Omega}
\rangle_{\mathbf L^2(\Gamma_{\boldsymbol\Omega})}
}_{\text{boundary flux terms}}
+\underbrace{
\frac1h
\langle
\boldsymbol\gamma_\partial\mathbf w,\mathbf v
\rangle_{\mathbf L^2(\Gamma_{\boldsymbol\Omega})}
}_{\text{boundary penalty term}},\nonumber
\end{align}
and the associated linear functional is
\begin{align}
\ell_h(\mathbf v)
:={}&
(\mathbf f,\mathbf v)_{\mathbf L^2(\boldsymbol\Omega)}
+\langle\psi_\Gamma,\avg{\mathbf v}\rangle_{L^2(\Gamma)}
-\langle
g_\Gamma,
\avg{\mathbf A\nabla\mathbf v}\cdot\vec n_\Gamma
\rangle_{L^2(\Gamma)}
+\frac{\gamma_\Gamma}{h}
\langle g_\Gamma,\jump{\mathbf v}\rangle_{L^2(\Gamma)}\nonumber
\\
&-\langle
\mathbf g_D,
\mathbf A\nabla\mathbf v\cdot\vec n_{\partial\Omega}
\rangle_{\mathbf L^2(\Gamma_{\boldsymbol\Omega})}
+\frac1h
\langle
\boldsymbol\gamma_\partial\mathbf g_D,\mathbf v
\rangle_{\mathbf L^2(\Gamma_{\boldsymbol\Omega})}. \label{eq:nitsche-linear}
\end{align}
The symmetry of $\mathbf A$ makes $a_h(\mathbf w,\mathbf v)$ symmetric.
The relation between $h$ and the kernel center sets is specified in
Section~\ref{sec:trace-inverse}.
For the kernel trial spaces, sufficient lower bounds on the
penalty parameters are established in Section~\ref{sec:analysis}.
The analysis assumes exact geometry and exact integration.

Consistency holds under the weaker regularity assumption stated below.
\begin{lemma}[Consistency]
\label{lem:consistency}
Let $\mathbf u\in\mathbf H^{1+\delta}(\boldsymbol\Omega)$,
$\delta>1/2$, solve \eqref{eq:intro-interface-problem}.
Then
\begin{equation}
a_h(\mathbf u,\mathbf v_h)=\ell_h(\mathbf v_h)
\qquad
\forall\,\mathbf v_h\in\mathbf V_h.
\label{eq:consistency}
\end{equation}
Consequently, any solution $\mathbf u_h$ of
\eqref{eq:discrete-nitsche} satisfies Galerkin orthogonality:
\begin{equation}
a_h(\mathbf u-\mathbf u_h,\mathbf v_h)=0
\qquad
\forall\,\mathbf v_h\in\mathbf V_h.
\label{eq:galerkin-orthogonality}
\end{equation}
\end{lemma}
\begin{proof}
The assumed regularity and the equations
$-\nabla\!\cdot(A^\eta\nabla u^\eta)=f^\eta\in L^2(\Omega^\eta)$
justify Green's formula with $L^2$ conormal traces
\cite[Corollary~6.11]{Mikhailov2013}.
For $\mathbf v_h\in\mathbf V_h$, integration by parts on the
subdomains and the prescribed flux jump give
\begin{align*}
(\mathbf A\nabla\mathbf u,\nabla\mathbf v_h)_{\mathbf L^2(\boldsymbol\Omega)}
={}&
(\mathbf f,\mathbf v_h)_{\mathbf L^2(\boldsymbol\Omega)}
+\langle\psi_\Gamma,\avg{\mathbf v_h}\rangle_{L^2(\Gamma)}
\\
&+
\langle
\avg{\mathbf A\nabla\mathbf u}\cdot\vec n_\Gamma,
\jump{\mathbf v_h}
\rangle_{L^2(\Gamma)}
+
\langle
\mathbf A\nabla\mathbf u\cdot\vec n_{\partial\Omega},
\mathbf v_h
\rangle_{\mathbf L^2(\Gamma_{\boldsymbol\Omega})}.
\end{align*}
Substituting this identity and the prescribed value conditions
into \eqref{eq:nitsche-bilinear} yields \eqref{eq:consistency}.
Galerkin orthogonality follows by subtracting
\eqref{eq:discrete-nitsche}.
\end{proof}

\subsection{Trace control required for stability}\label{subsec:stability-requirement}
Consistency alone does not ensure stability of the discrete
problem. The boundary and interface flux terms have no fixed
sign and must be controlled uniformly by the diffusion and
penalty terms.
We identify the discrete trace estimate that provides this
control and will verify it for kernel spaces in Section~
\ref{sec:trace-inverse}.

To establish stability of \eqref{eq:discrete-nitsche}, we seek
the continuity and coercivity bounds
\begin{equation}
|a_h(\mathbf w_h,\mathbf v_h)|
\le C\|\mathbf w_h\|_h\|\mathbf v_h\|_h,\quad
a_h(\mathbf v_h,\mathbf v_h)
\ge c\|\mathbf v_h\|_h^2,
\quad
\mathbf w_h,\mathbf v_h\in\mathbf V_h,
\label{eq:nitsche-stability-goal}
\end{equation}
with constants $C,c>0$ independent of $h$.
The coercivity bound will be established for sufficiently large
fixed penalty parameters.
Here the Nitsche energy norm is
\begin{equation}
\|\mathbf v\|_h^2
:=
\|\nabla\mathbf v\|_{\mathbf L^2(\boldsymbol\Omega)}^2
+h^{-1}\|\jump{\mathbf v}\|_{L^2(\Gamma)}^2
+h^{-1}\|\mathbf v\|_{\mathbf L^2(\Gamma_{\boldsymbol\Omega})}^2,
\qquad
\mathbf v\in\mathbf H^1(\boldsymbol\Omega).
\label{eq:nitsche-energy-norm}
\end{equation}
The subdomain gradients control each component only up to a
constant. The following Poincar\'e inequality controls these
constants through the interface jump and the exterior trace on
$\Gamma_{\Omega^+}$.

\begin{lemma}[Poincar\'e inequality with interface terms]
\label{lem:interface-poincare}
Under the geometric assumptions of
Subsection~\ref{subsec:interface-setting} with
$|\Gamma|_{d-1}>0$ and
$|\Gamma_{\Omega^+}|_{d-1}>0$, there exists $C>0$ such that every
$\mathbf v\in\mathbf H^1(\boldsymbol\Omega)$ satisfies
\begin{equation}
\|\mathbf v\|_{\mathbf L^2(\boldsymbol\Omega)}^2
\le C\left(
\|\nabla\mathbf v\|_{\mathbf L^2(\boldsymbol\Omega)}^2
+\|\jump{\mathbf v}\|_{L^2(\Gamma)}^2
+\|v^+\|_{\mathbf L^2(\Gamma_{\Omega^+})}^2
\right).
\label{eq:interface-poincare}
\end{equation}
\end{lemma}

\begin{proof}
For a bounded connected Lipschitz domain $D$ and a measurable
boundary portion $\Sigma$ of positive surface measure, the
Poincar\'e--Friedrichs inequality with a boundary term
\cite{Brenner2003,BrennerScott2008} gives
\begin{equation}
\|z\|_{L^2(D)}
\le C\left(\|\nabla z\|_{L^2(D)}+\|z\|_{L^2(\Sigma)}\right),
\qquad z\in H^1(D).
\label{eq:PF-boundary}
\end{equation}
Applying the Poincar\'e-Friedrichs inequality~\eqref{eq:PF-boundary} with
$(D,\Sigma)=(\Omega^+,\Gamma_{\Omega^+})$ gives
\begin{equation}
\|v^+\|_{L^2(\Omega^+)}
\le C\left(
\|\nabla v^+\|_{L^2(\Omega^+)}
+\|v^+\|_{L^2(\Gamma_{\Omega^+})}
\right).
\label{eq:PF-plus}
\end{equation}
The continuous trace inequality \cite{BrennerScott2008} and \eqref{eq:PF-plus} give
\begin{align*}
\|v^+\|_{L^2(\Gamma)} \le
C
\|v^+\|_{L^2(\Omega^+)}^{\frac 12}
\| v^+\|_{H^1(\Omega^+)}^{\frac 12}
&\le C\left(
\|v^+\|_{L^2(\Omega^+)}
+\|\nabla v^+\|_{L^2(\Omega^+)}
\right)\\
&\le C\left(
\|\nabla v^+\|_{L^2(\Omega^+)}
+\|v^+\|_{L^2(\Gamma_{\Omega^+})}
\right).
\end{align*}
Since $v^-|_\Gamma=\jump{\mathbf v}+v^+|_\Gamma$, we have
$\|v^-\|_{L^2(\Gamma)}
\le \|\jump{\mathbf v}\|_{L^2(\Gamma)}+\|v^+\|_{L^2(\Gamma)}$.
Applying the Poincar\'e-Friedrichs inequality~\eqref{eq:PF-boundary} with
$(D,\Sigma)=(\Omega^-,\Gamma)$ therefore yields
\[
\|v^-\|_{L^2(\Omega^-)}
\le C\left(
\|\nabla\mathbf v\|_{\mathbf L^2(\boldsymbol\Omega)}
+\|\jump{\mathbf v}\|_{L^2(\Gamma)}
+\|v^+\|_{L^2(\Gamma_{\Omega^+})}
\right).
\]
Squaring the two subdomain estimates and adding them proves
\eqref{eq:interface-poincare}.
\end{proof}

Since $h\le1$, the interface Poincar\'e inequality
\eqref{eq:interface-poincare} implies
$\|\mathbf v\|_{\mathbf H^1(\boldsymbol\Omega)}
\le C\|\mathbf v\|_h$ for every
$\mathbf v\in\mathbf H^1(\boldsymbol\Omega)$.
The diffusion and penalty terms are bounded in this norm.
Control of the remaining flux terms requires a discrete estimate.
For a single-domain Dirichlet problem, Chernih et al.~\cite[Section~4]{ChernihLeGia2014} assume a boundary
normal-derivative estimate in terms of the volume gradient
in their kernel Galerkin stability analysis.
For the present interface formulation, we seek to establish
the sufficient discrete estimate
\begin{equation}
h^{1/2}\left(
\left\|\avg{\mathbf A\nabla\mathbf v_h}\cdot\vec n_\Gamma
\right\|_{L^2(\Gamma)}
+\left\|\mathbf A\nabla\mathbf v_h\cdot\vec n_{\partial\Omega}
\right\|_{\mathbf L^2(\Gamma_{\boldsymbol\Omega})}
\right)
\le C_{\rm tr}\|\mathbf v_h\|_h,
\quad \mathbf v_h\in\mathbf V_h,
\label{eq:discrete-conormal-control}
\end{equation}
where $C_{\rm tr}$ is independent of $h$ and the penalty parameters.
This estimate serves the same purpose in the stability analysis,
while accounting for the averaged conormal flux on the interface
and the conormal fluxes on the exterior boundary portions
of the two subdomains.

If \eqref{eq:discrete-conormal-control} holds,
Cauchy--Schwarz and Young's inequalities give continuity and,
for sufficiently large penalty parameters independent of
refinement, coercivity, as detailed in
Theorem~\ref{thm:well-posedness}.
For the kernel product spaces introduced in
Section~\ref{sec:trace-inverse},
Theorem~\ref{thm:conormal-trace} proves
\eqref{eq:discrete-conormal-control} under
Assumption~\ref{ass:product-kernel} and the coefficient
conditions of Subsection~\ref{subsec:interface-setting}.
That section also establishes the approximation estimates
used in the error analysis.

\section{Kernel trial spaces and trace inverse estimates}
\label{sec:trace-inverse}
We now specialize the product trial space of
Section~\ref{sec:setting} to independently generated
Sobolev-kernel spaces.
For these spaces, we establish the sufficient discrete
conormal estimate \eqref{eq:discrete-conormal-control}
and construct comparison functions with approximation
bounds in the Nitsche norms.
The geometric and coefficient assumptions of
Subsection~\ref{subsec:interface-setting} remain in force.

\subsection{Sobolev kernels and trial spaces}
\label{subsec:kernel-setting}

For a domain $D\subset\mathbb R^d$ and a set
$X_D=\{\vec x_1,\ldots,\vec x_{N_D}\}\subset D$ of
$N_D\ge2$ distinct trial centers, define the fill distance, separation
radius, and mesh ratio by
\begin{equation}
h_D:=\sup_{\vec x\in D}\min_{\vec x_j\in X_D}|\vec x-\vec x_j|,
\qquad
q_D:=\frac12\min_{i\ne j}|\vec x_i-\vec x_j|,
\qquad
\rho_D:=\frac{h_D}{q_D}.
\label{eq:center-scales}
\end{equation}

\begin{assumption}[Sobolev kernel setting]
\label{ass:sobolev-kernel}
The domain $D$ is bounded and Lipschitz, and the center set $X_D$
satisfies $0<q_D\le1$.
The kernel $\phi_m(\vec x,\vec y)=\Phi_m(\vec x-\vec y)$
is real-valued, continuous, and strictly positive definite.
The Fourier transform of $\Phi_m$ is represented by a function satisfying
\begin{equation}
c_\phi(1+|\vec\xi|^2)^{-m}
\le\widehat\Phi_m(\vec\xi)
\le C_\phi(1+|\vec\xi|^2)^{-m},
\qquad \vec\xi\in\mathbb R^d,
\label{eq:sobolev-kernel}
\end{equation}
for some Sobolev smoothness order $m>d/2$ and constants $0<c_\phi\le C_\phi<\infty$.
\end{assumption}

\begin{assumption}[Product kernel discretization]
\label{ass:product-kernel}
On each subdomain of Subsection~\ref{subsec:interface-setting},
Assumption~\ref{ass:sobolev-kernel} holds with
$D=\Omega^\eta$ and $X_D=X^\eta$, $\eta\in\{-,+\}$.
The same kernel is used in both subdomains, with $m\ge2$.
The kernel, including its order and scale, is fixed under refinement.
Write $h_\eta$, $q_\eta$, and $\rho_\eta$ for the parameters of $X^\eta$.
The center families are quasi-uniform with comparable fill distances:
\begin{equation}
\rho_\eta\le\rho_0,
\quad
c_hh\le h_\eta\le h,
\quad
h:=\max\{h_-,h_+\}\le h_0\le1, \quad q:=\min\{q_-,q_+\},
\label{eq:center-assumptions}
\end{equation}
where $\rho_0$ and $0<c_h\le1$ are independent of refinement,
and $h_0$ is sufficiently small.
\end{assumption}

These assumptions give $q_\eta\simeq h_\eta\simeq h$, with
comparison constants independent of refinement, and identify the
scale $h$ in the Nitsche formulation.
No centers are required on the interface or exterior boundary.

Let $\mathcal N_{\phi_m}$ denote the native space, the reproducing
kernel Hilbert space associated with $\phi_m$.
Under Assumption~\ref{ass:sobolev-kernel}, the Fourier bounds
\eqref{eq:sobolev-kernel} imply
\begin{equation}
\mathcal N_{\phi_m}(\mathbb R^d)\simeq H^m(\mathbb R^d),
\qquad
\mathcal N_{\phi_m}(D)\simeq H^m(D),
\label{eq:native-sobolev-equivalence}
\end{equation}
where $\simeq$ denotes equality of spaces with equivalent norms
\cite{Wendland2005,SchabackWendland2006}.
Point evaluation is continuous since $m>d/2$.

The numerical experiments use two kernel families satisfying
\eqref{eq:sobolev-kernel}.
The globally supported Mat\'ern kernel is
\begin{equation}
\phi^{\mathrm M}_{m,\varepsilon}(\vec x,\vec y)
=\frac{2^{1-\nu}}{\Gamma(\nu)}z^\nu K_\nu(z),
\qquad
z=\varepsilon|\vec x-\vec y|,
\quad \nu=m-\frac d2>0,
\label{eq:matern-kernel}
\end{equation}
where $\varepsilon>0$ is the scale parameter and $K_\nu$ is the
modified Bessel function of the second kind.
The value at $z=0$ is defined by continuity.
For $d=2,3$, the compactly supported Wendland kernels satisfy
\begin{equation}
\phi^{\mathrm W}_{d,k,\varepsilon}(\vec x,\vec y)
=\psi_{d,k}\left(\varepsilon |\vec x-\vec y|\right),
\qquad
\mathcal N_{\phi^{\mathrm W}_{d,k,\varepsilon}}(\mathbb R^d)
\simeq H^{d/2+k+1/2}(\mathbb R^d),
\label{eq:wendland-kernel}
\end{equation}
where the support radius is $\varepsilon^{-1}$ and $m=d/2+k+1/2$.
Table~\ref{tab:wendland} lists representative profiles
\cite{Wendland1995,Wendland2005}.
The requirement $m\ge2$ excludes $k=0$ when $d=2$.

\begin{table}[tbp]
\centering
\caption{Normalized Wendland profiles for $d=2,3$, with
$(1-r)_+:=\max\{1-r,0\}$.}
\label{tab:wendland}
\footnotesize
\setlength{\tabcolsep}{8pt}
\renewcommand{\arraystretch}{1.2}
\begin{tabular}{cl@{\hspace{18pt}}cl}
\toprule
$k$ & $\psi_{d,k}(r)$ & $k$ & $\psi_{d,k}(r)$\\
\midrule
$0$ & $(1-r)_+^2$
& $2$ & $\frac13(1-r)_+^6(35r^2+18r+3)$\\
$1$ & $(1-r)_+^4(4r+1)$
& $3$ & $(1-r)_+^8(32r^3+25r^2+8r+1)$\\
\bottomrule
\end{tabular}
\end{table}

For a fixed kernel, the translates centered at $X_D$ span
\begin{equation}
V_{X_D,\Phi_m,D}
:=\operatorname{span}
\left\{\Phi_m(\cdot-\vec x_j)|_D:1\le j\le N_D\right\}
\subset\mathcal N_{\phi_m}(D).
\label{eq:kernel-trial-space}
\end{equation}
Strict positive definiteness makes these translates a basis of
dimension $N_D$.
For $\vec c=(c_1,\ldots,c_{N_D})^\top$, define
$v_{\vec c}:=\sum_{j=1}^{N_D}c_j\phi_m(\cdot,\vec x_j)|_D$
and $G_D:=\bigl(\phi_m(\vec x_i,\vec x_j)\bigr)_{i,j=1}^{N_D}$.
The kernel matrix $G_D$ is symmetric positive definite, and the
reproducing property gives
$\|v_{\vec c}\|_{\mathcal N_{\phi_m}(D)}^2=\vec c^\top G_D\vec c$.
For every continuous target $w$, the interpolant
$I_{X_D}w\in V_{X_D,\Phi_m,D}$ is uniquely determined by
$I_{X_D}w(\vec x_i)=w(\vec x_i)$, $1\le i\le N_D$.
Its coefficients solve $G_D\vec c=\vec w_D$, where
$\vec w_D:=\bigl(w(\vec x_i)\bigr)_{i=1}^{N_D}$.
Thus the translate coefficients are generally different from
the nodal values.

For the interface problem, under Assumption~\ref{ass:product-kernel},
set $V_h^\eta:=V_{X^\eta,\Phi_m,\Omega^\eta}$ for $\eta\in\{-,+\}$.
The product trial space of Subsection~\ref{subsec:nitsche-formulation}
then becomes
\begin{equation}
\mathbf V_h
:=\prod_{\eta\in\{-,+\}}V_{X^\eta,\Phi_m,\Omega^\eta}
=V_h^-\times V_h^+
\subset\mathbf H^m(\boldsymbol\Omega).
\label{eq:product-kernel-space}
\end{equation}
The components satisfy no boundary or interface value constraints.
Their conormal traces belong to $L^2$ since $m\ge2$.
Let $\{\boldsymbol\varphi_i\}_{i=1}^{N}$, with
$N=|X^-|+|X^+|$, be the basis obtained by placing each translate
in its subdomain component and setting the other component to zero.
Writing $\mathbf u_h=\sum_{j=1}^{N}c_j\boldsymbol\varphi_j$
gives the Galerkin system \cite{Wendland1999}
\begin{equation}
K_h\vec c=\vec b,
\qquad
(K_h)_{ij}:=a_h(\boldsymbol\varphi_j,\boldsymbol\varphi_i),
\qquad
b_i:=\ell_h(\boldsymbol\varphi_i).
\label{eq:kernel-galerkin-system}
\end{equation}
The stiffness matrix $K_h$ is distinct from the interpolation
matrices $G_{\Omega^\eta}$.
For sufficiently large penalty parameters, its positive definiteness
follows from the coercivity proved in Section~\ref{sec:analysis}.

Approximation in these spaces depends on the kernel order $m$
and the available target regularity $s$.
Under Assumption~\ref{ass:sobolev-kernel}, let
$2\le s\le m$ with $s>d/2$.
For sufficiently small $h_D$, every $w\in H^s(D)$ satisfies
\begin{equation}
\|w-I_{X_D}w\|_{H^j(D)}
\le C h_D^{s-j}\rho_D^{m-s}\|w\|_{H^s(D)},
\quad
w\in H^s(D),\quad j\in\mathbb N_0,\quad j\le s,
\label{eq:kernel-interpolation}
\end{equation}
with $C$ independent of $X_D$ and $w$
\cite[Appendix~C]{WynneBriolGirolami2021}.
In particular, $s<m$ permits targets outside the native space.
For integer $s$, the endpoint $j=s$ gives a stability bound. For the same range of $s$,
define the product interpolant componentwise by
$(I_h\mathbf w)^\eta:=I_{X^\eta}w^\eta$ for $\eta\in\{-,+\}$.
Under Assumption~\ref{ass:product-kernel}, the bounded mesh ratios
and $h_\eta\le h$ give
\begin{equation}
\|\mathbf w-I_h\mathbf w\|_{\mathbf H^j(\boldsymbol\Omega)}
\le C h^{s-j}\|\mathbf w\|_{\mathbf H^s(\boldsymbol\Omega)},
\quad\mathbf w\in\mathbf H^s(\boldsymbol\Omega), \quad j\in\mathbb N_0,\quad j\le s\le m.
\label{eq:product-interpolation}
\end{equation}

\begin{remark}[Kernel order and refinement]
\label{rem:kernel-refinement}
\textit{For $j=0,1$, the order $h^{s-j}$ agrees with the standard
approximation order of degree-$p$ finite elements when $p+1\ge s$
\cite{BrennerScott2008}.
Here $m$ determines the native-space smoothness, while $s$ records
the regularity used in the estimate.
The limit $h\to0$ refines the centers at fixed kernel order and scale.
Changing $m$ on a fixed center set changes the basis functions but
leaves the dimension unchanged, unlike finite element $p$-refinement.
Equal dimensions do not imply equal computational cost.
The effects of kernel support, and conditioning
are examined numerically in Subsection~\ref{subsec:example1} and \ref{subsec:example3}.}
\end{remark}

The product space has independent interface traces and no
prescribed boundary values. We next show that its conormal
traces nevertheless satisfy the control required by the
Nitsche formulation.

\subsection{Product-space trace inverse estimates}
\label{subsec:kernel-trace}
For a single-domain Dirichlet problem, stability of the symmetric
Nitsche method requires control of the boundary flux terms by
the volume diffusion and boundary penalty terms.
In the kernel Galerkin setting, Chernih et al.
\cite[Section~4]{ChernihLeGia2014} base their stability analysis
on an assumed estimate for the boundary normal derivative
in terms of the volume gradient.
This motivates establishing the required flux control directly
from the kernel and center assumptions.
For the present spaces, a sufficient trace inverse estimate is
\begin{equation}
\|\nabla v\|_{L^2(\partial D)}
\le Ch_D^{-1/2}
\left(\|\nabla v\|_{L^2(D)}+\|v\|_{L^2(\partial D)}\right),
\qquad v\in V_{X_D,\Phi_m,D}.
\label{eq:single-domain-trace}
\end{equation}
We develop the proof in the more general interface setting,
using independent kernel spaces on the two subdomains.
The resulting trace estimates yield the discrete conormal
control \eqref{eq:discrete-conormal-control}, and
\eqref{eq:single-domain-trace} is obtained as a corollary.

For finite elements on shape-regular, quasi-uniform fitted meshes,
boundary flux control follows from polynomial trace inverse
estimates for the elementwise gradients
\cite{BrennerScott2008,BeckerHansboStenberg2003}.
Kernel translates lack the elementwise polynomial structure used
in this argument, and their derivatives need not belong to the
original trial space.
Applying a continuous trace estimate to the gradient instead
introduces the $H^2(D)$ norm of the trial function.
Controlling this norm by its $H^1(D)$ norm requires a
bounded-domain inverse estimate at low Sobolev orders.
The Bernstein inequality of 
\cite[Lemma~3.2]{CheungLingSchaback2018} requires the lower
Sobolev order to exceed $d/2$ and does not directly cover this case.
The recent full-range result of \cite[Theorem~4.6]{SunLing2025} provides the required
$H^2$--$H^1$ estimate, together with the $H^1$--$L^2$ estimate
used below for value traces.

\begin{lemma}[Full-range Bernstein inequality
{\cite[Theorem~4.6]{SunLing2025}}]
\label{lem:full-range-bernstein}
Under Assumption~\ref{ass:sobolev-kernel}, for every pair of real
numbers $0\le t\le\alpha\le m$, there exists a constant $C>0$,
independent of $X_D$, such that
\begin{equation}
\|v\|_{H^\alpha(D)}
\le Cq_D^{t-\alpha}\|v\|_{H^t(D)},
\qquad v\in V_{X_D,\Phi_m,D}.
\label{eq:full-range-bernstein}
\end{equation}
\end{lemma}

For the interface problem, applying
\eqref{eq:single-domain-trace} separately on the subdomains would
introduce individual interface values, whereas the Nitsche penalty
controls their jump.
We instead use the interface Poincar\'e inequality Lemma~\ref{lem:interface-poincare} to control the
subdomain $L^2$ terms in the $H^1$ norms supplied by
Lemma~\ref{lem:full-range-bernstein}.
We write $\partial\boldsymbol\Omega:=
(\partial\Omega^-,\partial\Omega^+)$ and
$\|\mathbf v\|_{\mathbf L^2(\partial\boldsymbol\Omega)}^2
:=\sum_{\eta=\pm}\|v^\eta\|_{L^2(\partial\Omega^\eta)}^2$.
Then, combined with the trace inequality \cite[Theorem~1.6.6]{BrennerScott2008}, we have
the following product-space estimates.

\begin{theorem}[Trace inverse estimates]
\label{thm:trace-inverse}
Under Assumption~\ref{ass:product-kernel}, there exist constant $C>0$, such that every
$\mathbf v_h\in\mathbf V_h$ satisfies
\begin{align}
\|v_h^\eta\|_{L^2(\partial\Omega^\eta)}
&\le Ch^{-1/2}\|v_h^\eta\|_{L^2(\Omega^\eta)},
\qquad \qquad \text{for}~ \eta\in\{-,+\},
\label{eq:function-trace-inverse}\\
\|\nabla \mathbf v_h\|_{\mathbf L^2(\partial \boldsymbol \Omega)}
&\le Ch^{-1/2}\left(
\|\nabla\mathbf v_h\|_{\mathbf L^2(\boldsymbol\Omega)}
+\|\jump{\mathbf v_h}\|_{L^2(\Gamma)}
+\|v_h^+\|_{L^2(\Gamma_{\Omega^+})}
\right).
\label{eq:gradient-trace-inverse}
\end{align}
\end{theorem}

\begin{proof}
We first prove \eqref{eq:function-trace-inverse}.
For $\eta\in\{-,+\}$, we have
$V_h^\eta\subset H^m(\Omega^\eta)\subset H^1(\Omega^\eta)$.
The trace inequality
\cite[Theorem~1.6.6]{BrennerScott2008} and
Lemma~\ref{lem:full-range-bernstein} with $(\alpha,t)=(1,0)$ give
\begin{align*}
\|v_h^\eta\|_{L^2(\partial\Omega^\eta)}
&\le C\|v_h^\eta\|_{L^2(\Omega^\eta)}^{1/2}
        \|v_h^\eta\|_{H^1(\Omega^\eta)}^{1/2}
\le Cq_\eta^{-1/2}\|v_h^\eta\|_{L^2(\Omega^\eta)}.
\end{align*}
Quasi-uniformity and comparable fill distances imply
$q_\eta\simeq h_\eta\simeq h$, which proves
\eqref{eq:function-trace-inverse}.
We next prove \eqref{eq:gradient-trace-inverse}.
Since $m\ge2$, we have $V_h^\eta\subset H^2(\Omega^\eta)$ and
$\nabla v_h^\eta\in[H^1(\Omega^\eta)]^d$.
Applying the trace inequality \cite[Theorem~1.6.6]{BrennerScott2008} and using
Cauchy--Schwarz inequality gives
\begin{align}
\|\nabla v_h^\eta\|_{L^2(\partial\Omega^\eta)}
&\le C\|\nabla v_h^\eta\|_{L^2(\Omega^\eta)}^{1/2}
        \|\nabla v_h^\eta\|_{H^1(\Omega^\eta)}^{1/2}
\le C\|\nabla v_h^\eta\|_{L^2(\Omega^\eta)}^{1/2}
        \|v_h^\eta\|_{H^2(\Omega^\eta)}^{1/2}.
\label{eq:grad-trace-step}
\end{align}
The full-range Bernstein inequality Lemma~\ref{lem:full-range-bernstein} with $(\alpha,t)=(2,1)$
gives
$\|v_h^\eta\|_{H^2(\Omega^\eta)}
\le Cq_\eta^{-1}\|v_h^\eta\|_{H^1(\Omega^\eta)}$.
Substituting this bound into \eqref{eq:grad-trace-step} yields
\begin{align}
\|\nabla v_h^\eta\|_{L^2(\partial\Omega^\eta)}
&\le Cq_\eta^{-1/2}
\|\nabla v_h^\eta\|_{L^2(\Omega^\eta)}^{1/2}
\|v_h^\eta\|_{H^1(\Omega^\eta)}^{1/2}\nonumber\\
&\le Ch^{-1/2}\left(
\|\nabla v_h^\eta\|_{L^2(\Omega^\eta)}
+\|v_h^\eta\|_{L^2(\Omega^\eta)}
\right).
\label{eq:local-gradient-trace}
\end{align}
Summing over the subdomains with $\eta=\pm$ gives
\begin{align*}
\|\nabla\mathbf v_h\|_{\mathbf L^2(\partial\boldsymbol\Omega)}
\le Ch^{-1/2}\left(
\|\nabla\mathbf v_h\|_{\mathbf L^2(\boldsymbol\Omega)}
+\|\mathbf v_h\|_{\mathbf L^2(\boldsymbol\Omega)}
\right).
\end{align*}
The interface Poincar\'e inequality Lemma~\ref{lem:interface-poincare} controls the last volume term.
\end{proof}

\begin{corollary}[Single-domain trace inverse estimate]
\label{cor:single-domain-trace}
Under the assumptions of Theorem~\ref{thm:trace-inverse},
let $D=\Omega^\eta$, $X_D=X^\eta$, and $h_D=h_\eta$
for $\eta\in\{-,+\}$.
There exists a constant $C>0$, such that \eqref{eq:single-domain-trace} holds.
\end{corollary}

\begin{proof}
By \eqref{eq:local-gradient-trace} and $h_D\le h$,
we have
$\|\nabla v\|_{L^2(\partial D)}
\le Ch_D^{-1/2}
(\|\nabla v\|_{L^2(D)}+\|v\|_{L^2(D)})$.
The boundary Poincar\'e--Friedrichs inequality
\eqref{eq:PF-boundary} with $\Sigma=\partial D$
bounds $\|v\|_{L^2(D)}$ by
$C(\|\nabla v\|_{L^2(D)}+\|v\|_{L^2(\partial D)})$
and proves \eqref{eq:single-domain-trace}.
\end{proof}

We next obtain the conormal estimate required in
Section~\ref{sec:setting}.

\begin{theorem}[Discrete conormal control]
\label{thm:conormal-trace}
Under Assumption~\ref{ass:product-kernel} and the coefficient
conditions of Subsection~\ref{subsec:interface-setting},
there exists $C_{\rm tr}>0$, such that \eqref{eq:discrete-conormal-control} holds.
\end{theorem}

\begin{proof}
The arithmetic average and the ellipticity of diffusion tensor~\eqref{eq:uniform-ellipticity} give
\[
\left\|\avg{\mathbf A\nabla\mathbf v_h}\cdot\vec n_\Gamma
\right\|_{L^2(\Gamma)}
+\left\|\mathbf A\nabla\mathbf v_h\cdot\vec n_{\partial\Omega}
\right\|_{\mathbf L^2(\Gamma_{\boldsymbol\Omega})}
\le\sqrt{2}\,\alpha_1
\|\nabla\mathbf v_h\|_{\mathbf L^2(\partial\boldsymbol\Omega)}.
\]
By trace inverse inequality~\eqref{eq:gradient-trace-inverse},
\begin{align*}
h^{1/2}\|\nabla\mathbf v_h\|_{\mathbf L^2(\partial\boldsymbol\Omega)}
&\le C\left(
\|\nabla\mathbf v_h\|_{\mathbf L^2(\boldsymbol\Omega)}
+\|\jump{\mathbf v_h}\|_{L^2(\Gamma)}
+\|v_h^+\|_{L^2(\Gamma_{\Omega^+})}
\right)
\le C\|\mathbf v_h\|_h.
\end{align*}
The last inequality follows from $h\le1$ and the definition
\eqref{eq:nitsche-energy-norm}.
Combining the two bounds proves
\eqref{eq:discrete-conormal-control}.
\end{proof}

\subsection{Interpolation error estimates under Nitsche norms}
\label{subsec:kernel-approximation}
The discrete conormal estimate applies to trial functions,
whereas approximation errors generally lie outside $\mathbf V_h$.
We control their boundary and interface fluxes using the augmented
Nitsche norm, whose weights complement those of the value terms
in \eqref{eq:nitsche-bilinear}.
Their complementary weights motivate the augmented Nitsche norm
\cite{BeckerHansboStenberg2003}.
We include these flux contributions in the augmented Nitsche norm
\begin{equation}
\begin{aligned}
\|\mathbf w\|_{h,*}^2
:={}&\|\mathbf w\|_h^2
+h\|\avg{\mathbf A\nabla\mathbf w}\cdot\vec n_\Gamma\|_{L^2(\Gamma)}^2
+h\|\mathbf A\nabla\mathbf w\cdot\vec n_{\partial\Omega}
\|_{\mathbf L^2(\Gamma_{\boldsymbol\Omega})}^2,
\hspace{1mm} \mathbf w\in\mathbf H^2(\boldsymbol\Omega). \hspace{-3mm}
\end{aligned}
\label{eq:augmented-norm}
\end{equation}

The approximation argument should apply throughout the range
$2\le s\le m$ without imposing the additional restriction $s>d/2$.
Direct pointwise interpolation of an arbitrary $H^s(D)$ function
would require continuous point evaluation, which is not available
when $s\le d/2$.
The error analysis, however, only requires a comparison
function in the trial space with suitable approximation.
We therefore smooth each component before kernel interpolation.
This provides a single construction for both the primal and
adjoint solutions without requiring point values of the
original target.

To define the smoothing by a Fourier cutoff, we extend
each component to $\mathbb R^d$ while preserving its Sobolev
regularity.
For $D=\Omega^\eta$, fix a linear extension operator $E_D$
that is bounded from $H^t(D)$ to $H^t(\mathbb R^d)$ for
$0\le t\le m$ \cite{Rychkov1999}.
Choose a real-valued radial function
$\chi\in C_0^\infty(\mathbb R^d)$ satisfying
$\chi=1$ on $B_1(0)$ and $\chi=0$ outside $B_2(0)$.
Define
\begin{equation}
S_h^Dw
:=\left.\mathcal F^{-1}
\bigl[\chi(h\,\cdot)\widehat{E_Dw}\bigr]\right|_D,
\qquad
(\Pi_h\mathbf w)^\eta
:=I_{X^\eta}S_h^{\Omega^\eta}w^\eta,
\quad \eta\in\{-,+\}.
\label{eq:regularized-interpolant}
\end{equation}
Since $S_h^Dw\in H^m(D)$ and $m>d/2$, its point values and
kernel interpolant are well defined.
Thus $\Pi_h$ maps $\mathbf H^s(\boldsymbol\Omega)$ into
$\mathbf V_h$ for every $2\le s\le m$.
It interpolates the smoothed components and serves only as
a comparison operator in the error analysis.

\begin{lemma}[Interpolation error estimates under Nitsche norms]
\label{lem:kernel-approximation}
Under Assumption~\ref{ass:product-kernel} and the coefficient
assumptions of Subsection~\ref{subsec:interface-setting},
let $2\le s\le m$.
Then every $\mathbf w\in\mathbf H^s(\boldsymbol\Omega)$ satisfies
\begin{equation}
\|\mathbf w-\Pi_h\mathbf w\|_h
\le\|\mathbf w-\Pi_h\mathbf w\|_{h,*}
\le Ch^{s-1}\|\mathbf w\|_{\mathbf H^s(\boldsymbol\Omega)}.
\label{eq:kernel-nitsche-approximation}
\end{equation}
\end{lemma}

\begin{proof}
For $w\in H^s(D)$ and $0<h\le1$, the support properties of
$\chi$ give
\begin{align*}
(1+|\vec{\xi}|^2)^{j/2}|1-\chi(h\vec{\xi})|
&\le Ch^{s-j}(1+|\vec{\xi}|^2)^{s/2},
\qquad j=0,1,2,
\\
(1+|\vec{\xi} |^2)^{m/2}|\chi(h\vec{\xi})|
&\le Ch^{s-m}(1+|\vec{\xi}|^2)^{s/2}.
\end{align*}
Plancherel's identity and the boundedness of $E_D$ therefore yield
\begin{align}
\|w-S_h^Dw\|_{H^j(D)}
&\le Ch^{s-j}\|w\|_{H^s(D)},
\qquad j=0,1,2,
\label{eq:smoothing-error}\\
\|S_h^Dw\|_{H^m(D)}
&\le Ch^{s-m}\|w\|_{H^s(D)}.
\label{eq:smoothing-hm}
\end{align}
For each $\eta\in\{-,+\}$,
applying \eqref{eq:kernel-interpolation} to
$S_h^{\Omega^\eta}w^\eta\in H^m(\Omega^\eta)$ gives
\begin{align*}
\|w^\eta-(\Pi_h\mathbf w)^\eta\|_{H^j(\Omega^\eta)}
&\le
\|w^\eta-S_h^{\Omega^\eta}w^\eta\|_{H^j(\Omega^\eta)}
+C h_\eta^{m-j}
\|S_h^{\Omega^\eta}w^\eta\|_{H^m(\Omega^\eta)}
\\
&\le
C\bigl(h^{s-j}+h_\eta^{m-j}h^{s-m}\bigr)
\|w^\eta\|_{H^s(\Omega^\eta)},
\qquad j=0,1,2.
\end{align*}
Since $h_\eta\le h$ and $m-j\ge0$, we obtain
\begin{equation}
\|w^\eta-(\Pi_h\mathbf w)^\eta\|_{H^j(\Omega^\eta)}
\le Ch^{s-j}\|w^\eta\|_{H^s(\Omega^\eta)},
\qquad j=0,1,2.
\label{eq:regularized-Hj}
\end{equation}
The case $j=1$ controls the volume term in the energy norm.
For the value traces, the trace inequality
\cite[Theorem~1.6.6]{BrennerScott2008} and
\eqref{eq:regularized-Hj} with $j=0,1$ give
\begin{align*}
h^{-1/2}\|w^\eta-(\Pi_h\mathbf w)^\eta\|_{L^2(\partial\Omega^\eta)}
&\le
Ch^{-1/2}
\|w^\eta-(\Pi_h\mathbf w)^\eta\|_{L^2(\Omega^\eta)}^{1/2}
\|w^\eta-(\Pi_h\mathbf w)^\eta\|_{H^1(\Omega^\eta)}^{1/2}
\\
&\le
Ch^{s-1}\|w^\eta\|_{H^s(\Omega^\eta)}.
\end{align*}
This bounds the interface jump and exterior boundary terms
in $\|\mathbf w-\Pi_h\mathbf w\|_h$.
Applying the same trace inequality to the gradient and using
\eqref{eq:regularized-Hj} with $j=1,2$ yields
\begin{align*}
h^{1/2}\|\nabla (w^\eta-(\Pi_h\mathbf w)^\eta)\|_{L^2(\partial\Omega^\eta)}
&\le
Ch^{1/2}
\|w^\eta-(\Pi_h\mathbf w)^\eta\|_{H^1(\Omega^\eta)}^{1/2}
\|w^\eta-(\Pi_h\mathbf w)^\eta\|_{H^2(\Omega^\eta)}^{1/2}
\\
&\le
Ch^{s-1}\|w^\eta\|_{H^s(\Omega^\eta)}.
\end{align*}
The coefficient bound \eqref{eq:uniform-ellipticity} then controls
the conormal terms in \eqref{eq:augmented-norm}.
Combining these estimates over the two subdomains proves
\eqref{eq:kernel-nitsche-approximation}.
\end{proof}

The same construction applies to the homogeneous adjoint solution
under the piecewise $H^2$ regularity assumption
\eqref{eq:adjoint-regularity}, without requiring $2>d/2$.
Taking $s=2$ in Lemma~\ref{lem:kernel-approximation} gives
\begin{equation}
\|\mathbf z-\Pi_h\mathbf z\|_{h,*}
\le Ch\|\mathbf z\|_{\mathbf H^2(\boldsymbol\Omega)}
\le CC_{\rm reg}h\|\mathbf r\|_{\mathbf L^2(\boldsymbol\Omega)}.
\label{eq:adjoint-approximation}
\end{equation}
This supplies the additional factor $h$ in the duality argument
for Theorem~\ref{thm:L2-error}.

\section{Well-posedness, error analysis, and conditioning}
\label{sec:analysis}
The conormal and approximation estimates of Section~\ref{sec:trace-inverse} provide uniform flux control on the
trial space and approximation in the norms required by the
bilinear form. We first use these properties to establish
well-posedness and error bounds. We then
analyze the separate question of conditioning in the
kernel-translate basis, with explicit scale dependence.

\subsection{Well-posedness}
\label{subsec:stability}
The discrete conormal estimate
\eqref{eq:discrete-conormal-control} supplies the trace control
identified in Subsection~\ref{subsec:stability-requirement}.
It allows the flux terms to be absorbed by the diffusion and
penalty terms with sufficiently large $h$-independent penalty parameters.

\begin{theorem}[Well-posedness]
\label{thm:well-posedness}
Under Assumption~\ref{ass:product-kernel} and the coefficient
assumptions of Subsection~\ref{subsec:interface-setting},
there exists a constant $C_b>0$, such that
\begin{equation}
|a_h(\mathbf w_h,\mathbf v_h)|
\le C_b\|\mathbf w_h\|_h\|\mathbf v_h\|_h,
\qquad \mathbf w_h,\mathbf v_h\in\mathbf V_h.
\label{eq:discrete-continuity}
\end{equation}
Let
$\gamma_{\min}:=
\min\{\gamma_\Gamma,\gamma_{\partial,-},\gamma_{\partial,+}\}$.
If
\begin{equation}
\gamma_{\min}\ge\alpha_0+\frac{2C_{\rm tr}^2}{\alpha_0},
\label{eq:penalty-threshold}
\end{equation}
where $C_{\rm tr}$ is the constant in
Theorem~\ref{thm:conormal-trace}, then
\begin{equation}
a_h(\mathbf v_h,\mathbf v_h)
\ge c_s\|\mathbf v_h\|_h^2,
\qquad \mathbf v_h\in\mathbf V_h,
\qquad c_s=\alpha_0/2.
\label{eq:coercivity}
\end{equation}
Consequently, \eqref{eq:discrete-nitsche} has a unique solution,
and $K_h$ is symmetric positive definite.
\end{theorem}

\begin{proof}
We first prove continuity.
The upper ellipticity bound gives
\[
|(\mathbf A\nabla\mathbf w_h,\nabla\mathbf v_h)_{\mathbf L^2(\boldsymbol\Omega)}|
\le\alpha_1
\|\nabla\mathbf w_h\|_{\mathbf L^2(\boldsymbol\Omega)}
\|\nabla\mathbf v_h\|_{\mathbf L^2(\boldsymbol\Omega)}.
\]
For the interface flux term, Cauchy--Schwarz and
\eqref{eq:discrete-conormal-control} give
\begin{align*}
\left|
\left\langle
\avg{\mathbf A\nabla\mathbf w_h}\cdot\vec n_\Gamma,
\jump{\mathbf v_h}
\right\rangle_{L^2(\Gamma)}
\right|
&\le
\left(h^{1/2}
\|\avg{\mathbf A\nabla\mathbf w_h}\cdot\vec n_\Gamma\|_{L^2(\Gamma)}\right)
\left(h^{-1/2}\|\jump{\mathbf v_h}\|_{L^2(\Gamma)}\right)
\\
&\le C_{\rm tr}\|\mathbf w_h\|_h\|\mathbf v_h\|_h.
\end{align*}
The exterior boundary flux term is estimated in the same way.
Interchanging $\mathbf w_h$ and $\mathbf v_h$ bounds the
corresponding symmetric flux terms.
The interface Nitsche term satisfies
\begin{align*}
\frac{\gamma_\Gamma}{h}
|\langle\jump{\mathbf w_h},\jump{\mathbf v_h}\rangle_{L^2(\Gamma)}|
&\le\gamma_\Gamma \left( h^{-1/2}\|\jump{\mathbf w_h}\|_{L^2(\Gamma)}\right) \left(h^{-1/2}\|\jump{\mathbf v_h}\|_{L^2(\Gamma)}\right) \le \gamma_\Gamma\|\mathbf w_h\|_h\|\mathbf v_h\|_h.
\end{align*}
The exterior boundary penalty is bounded in the same way.
Combining these estimates proves \eqref{eq:discrete-continuity}.\hfill\break
To prove coercivity, take $\mathbf w_h=\mathbf v_h$ in
\eqref{eq:nitsche-bilinear}, we have
\begin{align}
a_h(\mathbf v_h,\mathbf v_h)
={}&
(\mathbf A\nabla\mathbf v_h,\nabla\mathbf v_h)_{\mathbf L^2(\boldsymbol\Omega)}
-2\left\langle
\avg{\mathbf A\nabla\mathbf v_h}\cdot\vec n_\Gamma,
\jump{\mathbf v_h}
\right\rangle_{L^2(\Gamma)} +\frac{\gamma_\Gamma}{h}\|\jump{\mathbf v_h}\|_{L^2(\Gamma)}^2
\nonumber\\
&-2\left\langle
\mathbf A\nabla\mathbf v_h\cdot\vec n_{\partial\Omega},
\mathbf v_h
\right\rangle_{\mathbf L^2(\Gamma_{\boldsymbol\Omega})}
+\frac1h
\langle\boldsymbol\gamma_\partial\mathbf v_h,\mathbf v_h
\rangle_{\mathbf L^2(\Gamma_{\boldsymbol\Omega})}.
\label{eq:coercivity-expansion}
\end{align}
Uniform ellipticity bounds the standard Galerkin term $(\mathbf A\nabla\mathbf v_h,\nabla\mathbf v_h)_{\mathbf L^2(\boldsymbol\Omega)}$ by
$\alpha_0\|\nabla\mathbf v_h\|_{\mathbf L^2(\boldsymbol\Omega)}^2$.
For the two flux contributions, Cauchy--Schwarz and
\eqref{eq:discrete-conormal-control} yield
\begin{align*}
2\left|
\left\langle
\avg{\mathbf A\nabla\mathbf v_h}\cdot\vec n_\Gamma,
\jump{\mathbf v_h}
\right\rangle_{L^2(\Gamma)}
\right|
&+2\left|
\left\langle
\mathbf A\nabla\mathbf v_h\cdot\vec n_{\partial\Omega},
\mathbf v_h
\right\rangle_{\mathbf L^2(\Gamma_{\boldsymbol\Omega})}
\right|
\\
&\le
2C_{\rm tr}\|\mathbf v_h\|_h
\left[
h^{-1}\left(
\|\jump{\mathbf v_h}\|_{L^2(\Gamma)}^2
+\|\mathbf v_h\|_{\mathbf L^2(\Gamma_{\boldsymbol\Omega})}^2
\right)
\right]^{1/2}.
\end{align*}
For any $\vartheta>0$, Young's inequality and
\eqref{eq:nitsche-energy-norm} give
\begin{align*}
&2C_{\rm tr}\|\mathbf v_h\|_h
\left[
h^{-1}\left(
\|\jump{\mathbf v_h}\|_{L^2(\Gamma)}^2
+\|\mathbf v_h\|_{\mathbf L^2(\Gamma_{\boldsymbol\Omega})}^2
\right)
\right]^{1/2}
\\
&\quad\le
\vartheta\|\nabla\mathbf v_h\|_{\mathbf L^2(\boldsymbol\Omega)}^2
+\left(\vartheta+\frac{C_{\rm tr}^2}{\vartheta}\right)
h^{-1}\left(
\|\jump{\mathbf v_h}\|_{L^2(\Gamma)}^2
+\|\mathbf v_h\|_{\mathbf L^2(\Gamma_{\boldsymbol\Omega})}^2
\right).
\end{align*}
Choose $\vartheta=\alpha_0/2$ and set
$C_0:=\alpha_0/2+2C_{\rm tr}^2/\alpha_0$.
Substitution into \eqref{eq:coercivity-expansion} gives
\begin{align}
a_h(\mathbf v_h,\mathbf v_h)
\ge{}&
\frac{\alpha_0}{2}
\|\nabla\mathbf v_h\|_{\mathbf L^2(\boldsymbol\Omega)}^2
+(\gamma_\Gamma-C_0)h^{-1}
\|\jump{\mathbf v_h}\|_{L^2(\Gamma)}^2
\nonumber\\
&+
(\gamma_{\partial,\eta}-C_0)h^{-1}
\|\mathbf v_h\|_{\mathbf L^2(\Gamma_{\boldsymbol \Omega})}^2.
\label{eq:coercivity-lower-bound}
\end{align}
Under \eqref{eq:penalty-threshold}, each penalty coefficient
in this bound is at least $\alpha_0/2$.
This proves \eqref{eq:coercivity}.\hfill\break
Since $a_h$ is symmetric and coercive on the finite-dimensional
space $\mathbf V_h$, its stiffness matrix is symmetric positive
definite, and \eqref{eq:discrete-nitsche} has a unique solution.
\end{proof}

For the error analysis, we also require continuity for arguments
outside the trial space.
The same termwise estimates, with each flux controlled by the
augmented norm \eqref{eq:augmented-norm}, give
\begin{equation}
|a_h(\mathbf w,\mathbf v)|
\le C_b\|\mathbf w\|_{h,*}\|\mathbf v\|_{h,*},
\qquad \mathbf w,\mathbf v\in\mathbf H^2(\boldsymbol\Omega),
\label{eq:augmented-continuity}
\end{equation}
after increasing $C_b$ if necessary.
Theorem~\ref{thm:conormal-trace} also gives
\begin{equation}
\|\mathbf v_h\|_h\le\|\mathbf v_h\|_{h,*}
\le C_N\|\mathbf v_h\|_h,
\qquad \mathbf v_h\in\mathbf V_h,
\qquad C_N=(1+C_{\rm tr}^2)^{1/2}.
\label{eq:discrete-norm-equivalence}
\end{equation}
Combining these bounds yields
\begin{equation}
|a_h(\mathbf w,\mathbf v_h)|
\le C_bC_N\|\mathbf w\|_{h,*}\|\mathbf v_h\|_h,
\qquad
\mathbf w\in\mathbf H^2(\boldsymbol\Omega),\quad
\mathbf v_h\in\mathbf V_h.
\label{eq:mixed-nitsche-continuity}
\end{equation}

\subsection{A priori error estimates}
\label{subsec:error-analysis}

Coercivity and Galerkin orthogonality reduce the energy error to
approximation in the augmented Nitsche norm.
The approximation estimates of
Subsection~\ref{subsec:kernel-approximation} then give convergence
rates determined by the piecewise solution regularity.

\begin{theorem}[Energy-norm error estimates]
\label{thm:energy-error}
Under Assumption~\ref{ass:product-kernel}, the coefficient
assumptions of Subsection~\ref{subsec:interface-setting}, and
the penalty condition \eqref{eq:penalty-threshold},
let  $\mathbf u\in\mathbf H^s(\boldsymbol\Omega)$ with
$2\le s\le m$ solve
\eqref{eq:intro-interface-problem}.
Then the discrete solution $\mathbf u_h$ satisfies
\begin{align}
\|\mathbf u-\mathbf u_h\|_h
&\le \|\mathbf u-\mathbf u_h\|_{h,*} \le Ch^{s-1}\|\mathbf u\|_{\mathbf H^s(\boldsymbol\Omega)}.
\label{eq:energy-error}
\end{align}
\end{theorem}

\begin{proof}
For any $\mathbf v_h\in\mathbf V_h$, coercivity
\eqref{eq:coercivity} and Galerkin orthogonality
\eqref{eq:galerkin-orthogonality} give
\begin{equation}
c_s\|\mathbf u_h-\mathbf v_h\|_h^2
\le a_h(\mathbf u_h-\mathbf v_h,\mathbf u_h-\mathbf v_h)
=a_h(\mathbf u-\mathbf v_h,\mathbf u_h-\mathbf v_h).
\label{eq:energy-cea-start}
\end{equation}
The mixed continuity estimate \eqref{eq:mixed-nitsche-continuity}
therefore yields
$\|\mathbf u_h-\mathbf v_h\|_h
\le\frac{C_bC_N}{c_s}\|\mathbf u-\mathbf v_h\|_{h,*}$.
The triangle inequality and the discrete norm equivalence
\eqref{eq:discrete-norm-equivalence} imply
\begin{equation}
\begin{aligned}
\|\mathbf u-\mathbf u_h\|_{h,*}
&\le\|\mathbf u-\mathbf v_h\|_{h,*}
+C_N\|\mathbf u_h-\mathbf v_h\|_h
\le\left(1+\frac{C_bC_N^2}{c_s}\right)
\|\mathbf u-\mathbf v_h\|_{h,*}.
\end{aligned}
\label{eq:energy-best-approximation}
\end{equation}
Taking $\mathbf v_h=\Pi_h\mathbf u$ and applying the
Nitsche-norm approximation estimate
Lemma~\ref{lem:kernel-approximation} proves the augmented-norm
bound in \eqref{eq:energy-error}.
The energy-norm bound follows from
$\|\cdot\|_h\le\|\cdot\|_{h,*}$.
\end{proof}

The $L^2$ estimate follows by duality, using the adjoint
approximation bound \eqref{eq:adjoint-approximation}.

\begin{theorem}[$L^2$ error estimate]
\label{thm:L2-error}
Under the assumptions of Theorem~\ref{thm:energy-error}, suppose additionally
that the adjoint regularity estimate
\eqref{eq:adjoint-regularity} holds. Let  $\mathbf u\in\mathbf H^s(\boldsymbol\Omega)$ with
$2\le s\le m$ solve
\eqref{eq:intro-interface-problem}.
Then 
\begin{equation}
\|\mathbf u-\mathbf u_h\|_{\mathbf L^2(\boldsymbol\Omega)}
\le Ch^s\|\mathbf u\|_{\mathbf H^s(\boldsymbol\Omega)}.
\label{eq:L2-error}
\end{equation}
\end{theorem}

\begin{proof}
Let $\mathbf z$ solve the homogeneous adjoint interface problem
\eqref{eq:adjoint-interface-problem} with
$\mathbf r=\mathbf u-\mathbf u_h$.
Integration by parts on the subdomains, symmetry of $\mathbf A$,
and the homogeneous transmission and boundary conditions give
\begin{align*}
\|\mathbf u-\mathbf u_h\|_{\mathbf L^2(\boldsymbol\Omega)}^2
=&
(\mathbf A\nabla(\mathbf u-\mathbf u_h),
 \nabla\mathbf z)_{\mathbf L^2(\boldsymbol\Omega)}
-
\left\langle
\jump{\mathbf u-\mathbf u_h},
\avg{\mathbf A\nabla\mathbf z}\cdot\vec n_\Gamma
\right\rangle_{L^2(\Gamma)}\\
&-
\left\langle
\mathbf u-\mathbf u_h,
\mathbf A\nabla\mathbf z\cdot\vec n_{\partial\Omega}
\right\rangle_{\mathbf L^2(\Gamma_{\boldsymbol\Omega})}=a_h(\mathbf u-\mathbf u_h,\mathbf z).
\end{align*}
Since $\Pi_h\mathbf z\in\mathbf V_h$, Galerkin orthogonality yields
$\|\mathbf u-\mathbf u_h\|_{\mathbf L^2(\boldsymbol\Omega)}^2
=a_h(\mathbf u-\mathbf u_h,\mathbf z-\Pi_h\mathbf z)$.
Using continuity in the augmented norm
\eqref{eq:augmented-continuity} and the adjoint approximation
bound \eqref{eq:adjoint-approximation} obtained from
Lemma~\ref{lem:kernel-approximation} with $s=2$, we obtain
\begin{align*}
\|\mathbf u-\mathbf u_h\|_{\mathbf L^2(\boldsymbol\Omega)}^2
\le
C_b\|\mathbf u-\mathbf u_h\|_{h,*}
\|\mathbf z-\Pi_h\mathbf z\|_{h,*}
&\le
Ch\|\mathbf u-\mathbf u_h\|_{h,*}
\|\mathbf z\|_{\mathbf H^2(\boldsymbol\Omega)}
\\
&\le
Ch\|\mathbf u-\mathbf u_h\|_{h,*}
\|\mathbf u-\mathbf u_h\|_{\mathbf L^2(\boldsymbol\Omega)}.
\end{align*}
If the error is nonzero, division gives
\begin{equation}
\|\mathbf u-\mathbf u_h\|_{\mathbf L^2(\boldsymbol\Omega)}
\le Ch\|\mathbf u-\mathbf u_h\|_{h,*}.
\label{eq:L2-duality-bound}
\end{equation}
The same bound holds trivially when the error vanishes.
Combining \eqref{eq:L2-duality-bound} with
\eqref{eq:energy-error} proves \eqref{eq:L2-error}.
\end{proof}

\subsection{Conditioning of the discrete system}
\label{subsec:conditioning}

The stability and error estimates concern functions in the
trial space.
The conditioning of the coefficient system also depends
on the kernel-translate basis.
The preceding results apply at any fixed positive kernel scale,
whereas the scale-explicit analysis below is restricted to
$0<\varepsilon\le1$.
Fix a reference kernel $\Phi_m$ satisfying
Assumption~\ref{ass:sobolev-kernel}, and set
$\Phi_{m,\varepsilon}(\vec x):=\Phi_m(\varepsilon\vec x)$ and
$\phi_{m,\varepsilon}(\vec x,\vec y)
:=\Phi_{m,\varepsilon}(\vec x-\vec y)$ for $0<\varepsilon\le1$.
Let $\mathbf V_{h,\varepsilon}$ and $K_{h,\varepsilon}$ denote
the space and stiffness matrix obtained by using this kernel in
\eqref{eq:product-kernel-space} and
\eqref{eq:kernel-galerkin-system}.
The center-family bounds in Assumption~\ref{ass:product-kernel}
are independent of $\varepsilon$.
The Fourier scaling identity
$\widehat\Phi_{m,\varepsilon}(\vec\xi)
=\varepsilon^{-d}\widehat\Phi_m(\vec\xi/\varepsilon)$ gives
\begin{equation}
c_\phi\varepsilon^{2m-d}(1+|\vec\xi|^2)^{-m}
\le\widehat\Phi_{m,\varepsilon}(\vec\xi)
\le C_\phi\varepsilon^{-d}(1+|\vec\xi|^2)^{-m}.
\label{eq:scaled-fourier-bounds}
\end{equation}

Mirzaei \cite{Mirzaei2026} studies scale-uniform inverse
inequalities in scaled and standard Sobolev norms for
$\delta^{-d}\Phi_m(\cdot/\delta)$ with $0<\delta\le1$.
Up to a scalar factor, our kernels correspond to
$\delta=\varepsilon^{-1}\ge1$.
For this range, we adapt the band-limited interpolation and
native-space argument in \cite{SunLing2025}, tracking the scale
dependence explicitly.
This gives the $H^1$-based inverse estimate needed below for
the condition-number analysis.

\begin{lemma}[Scale-uniform Bernstein estimate]
\label{lem:scaled-bernstein}
Under Assumption~\ref{ass:sobolev-kernel}, let $m\ge2$.
For $1\le\alpha\le m$, there exists $C>0$, independent of
$\varepsilon$, $X_D$, and $v$, such that
\begin{equation}
\|v\|_{H^\alpha(D)}
\le Cq_D^{1-\alpha}\|v\|_{H^1(D)},
\qquad v\in V_{X_D,\Phi_{m,\varepsilon},D}.
\label{eq:scaled-bernstein}
\end{equation}
\end{lemma}

\begin{proof}
Fix $r\in(\max\{1,d/2\},m)$.
By \cite[Lemma~4.3]{SunLing2025}, there is a band-limited
function $F_\sigma$ satisfying $F_\sigma|_{X_D}=v|_{X_D}$ and
$\sigma\simeq q_D^{-1}$, with
$\|F_\sigma\|_{H^m(\mathbb R^d)}
\le Cq_D^{r-m}\|v\|_{H^r(D)}$.
The Fourier bounds and the minimum-norm property of kernel
interpolation give
\begin{align*}
|v|_{H^m(D)}
&\le C\varepsilon^{m-d/2}
\|v\|_{\mathcal N_{\phi_{m,\varepsilon}}(D)}
\le C\varepsilon^{m-d/2}
\|F_\sigma\|_{\mathcal N_{\phi_{m,\varepsilon}}(\mathbb R^d)}
\\
&\le C\varepsilon^{m-d/2}\varepsilon^{d/2-m}
\|F_\sigma\|_{H^m(\mathbb R^d)}
\le Cq_D^{r-m}\|v\|_{H^r(D)}.
\end{align*}
The first inequality follows from
$|\vec\xi|^{2m}\widehat\Phi_{m,\varepsilon}(\vec\xi)
\le C_\phi\varepsilon^{2m-d}$, obtained by Fourier scaling
and the minimum-extension norm.
For noninteger $m$, $|\cdot|_{H^m(D)}$ denotes the highest-order
Sobolev--Slobodeckij seminorm.
On the fixed domain,
$\|v\|_{H^m(D)}\le C(\|v\|_D+|v|_{H^m(D)})$.
Since $q_D\le1$, Sobolev interpolation therefore yields
\[
\|v\|_{H^m(D)}
\le Cq_D^{r-m}
\|v\|_{H^1(D)}^{(m-r)/(m-1)}
\|v\|_{H^m(D)}^{(r-1)/(m-1)}.
\]
For $v\ne0$, division gives
$\|v\|_{H^m(D)}\le Cq_D^{1-m}\|v\|_{H^1(D)}$.
Interpolation between orders $1$ and $m$ proves
\eqref{eq:scaled-bernstein}.
\end{proof}

\begin{theorem}[Condition-number estimate]
\label{thm:condition-number}
Assume the geometric and coefficient conditions of
Subsection~\ref{subsec:interface-setting}.
Let the reference kernel $\Phi_m$ satisfy
Assumption~\ref{ass:sobolev-kernel} with $m\ge2$, and let the
center families satisfy \eqref{eq:center-assumptions}.
Choose penalty parameters independent of $h$ and $\varepsilon$
satisfying \eqref{eq:penalty-threshold}, with $C_{\rm tr}$ replaced
by a bound uniform in $\varepsilon$.
Then, for $0<\varepsilon\le1$, $K_{h,\varepsilon}$ is symmetric
positive definite and
\begin{equation}
\kappa_2(K_{h,\varepsilon})
\le C\varepsilon^{-(4m-2d)}h^{-(4m-1)}.
\label{eq:condition-number}
\end{equation}
\end{theorem}

\begin{proof}
Using the scale-uniform Bernstein estimate Lemma~\ref{lem:scaled-bernstein} with $\alpha=2$ in the
proof of the discrete conormal control Theorem~\ref{thm:conormal-trace} gives a bound on
$C_{\rm tr}$ independent of $\varepsilon$.
Theorem~\ref{thm:well-posedness} therefore gives continuity and
coercivity with constants independent of $h$ and $\varepsilon$.
For $\mathbf v_h\in\mathbf V_{h,\varepsilon}$, let $\vec c$
be its translate coefficient vector, with subdomain blocks
$\vec c^{\,\eta}$.
Let $G_{\Omega^\eta,\varepsilon}$ be the corresponding kernel matrix.
The Fourier lower bound and the standard kernel-matrix estimate
\cite[Chapter~12]{Wendland2005} give
$\lambda_{\min}(G_{\Omega^\eta,\varepsilon})
\ge c\varepsilon^{2m-d}q_\eta^{2m-d}$.
Choose a bounded extension operator
$E_\eta:H^m(\Omega^\eta)\to H^m(\mathbb R^d)$
\cite[Lemma~4.1]{SunLing2025}.
The native norm identity, its minimum-extension characterization,
and the Fourier lower bound give
\begin{align*}
c\varepsilon^{2m-d}q_\eta^{2m-d}
\|\vec c^{\,\eta}\|_2^2
&\le
(\vec c^{\,\eta})^\top
G_{\Omega^\eta,\varepsilon}\vec c^{\,\eta}
=
\|v_h^\eta\|_{\mathcal N_{\phi_{m,\varepsilon}}(\Omega^\eta)}^2
\le
\|E_\eta v_h^\eta\|_
{\mathcal N_{\phi_{m,\varepsilon}}(\mathbb R^d)}^2\\
&
\le C\varepsilon^{d-2m}
\|E_\eta v_h^\eta\|_{H^m(\mathbb R^d)}^2
\le C\varepsilon^{d-2m}
\|v_h^\eta\|_{H^m(\Omega^\eta)}^2.
\end{align*}
Applying the scale-uniform Bernstein estimate Lemma~\ref{lem:scaled-bernstein} with $\alpha=m$ yields
\begin{align*}
\|\vec c^{\,\eta}\|_2^2
&\le C\varepsilon^{-(4m-2d)}q_\eta^{d-2m}
\|v_h^\eta\|_{H^m(\Omega^\eta)}^2
\le C\varepsilon^{-(4m-2d)}q_\eta^{-(4m-d-2)}
\|v_h^\eta\|_{H^1(\Omega^\eta)}^2.
\end{align*}
Summing over the subdomains and using coercivity, the interface Poincar\'e inequality
Lemma~\ref{lem:interface-poincare}, and $q_\eta\simeq h$, we obtain
\begin{align*}
\vec c^\top K_{h,\varepsilon}\vec c
&=a_h(\mathbf v_h,\mathbf v_h)
\ge c_s\|\mathbf v_h\|_h^2
\ge c\|\mathbf v_h\|_{\mathbf H^1(\boldsymbol\Omega)}^2
\ge c\varepsilon^{4m-2d}h^{4m-d-2}\|\vec c\|_2^2.
\end{align*}
Consequently,
$\lambda_{\min}(K_{h,\varepsilon})
\ge c\varepsilon^{4m-2d}h^{4m-d-2}$.
For the upper bound, let
$\varphi_{j,\varepsilon}
:=\Phi_m(\varepsilon(\cdot-\vec x_j))|_D$,
where $D=\Omega^\eta$.
The Fourier decay gives $\Phi_m\in H^m(\mathbb R^d)$, and
positive definiteness gives $|\Phi_m(\vec x)|\le\Phi_m(\vec0)$.
Scaling and fixed-domain norm equivalence yield
\[
\|\varphi_{j,\varepsilon}\|_{H^m(D)}
\le C\left(
\Phi_m(\vec0)|D|^{1/2}
+\varepsilon^{m-d/2}|\Phi_m|_{H^m(\mathbb R^d)}
\right)
\le C.
\]
The trace theorem then gives
$\|\boldsymbol\varphi_{j,\varepsilon}\|_h^2\le Ch^{-1}$
for each product basis function.
Using continuity, Cauchy--Schwarz, and $N\le Ch^{-d}$, we obtain
\begin{align*}
\vec c^\top K_{h,\varepsilon}\vec c
&\le C_b\|\mathbf v_h\|_h^2
\le C\|\vec c\|_2^2
\sum_{j=1}^N\|\boldsymbol\varphi_{j,\varepsilon}\|_h^2
\le CNh^{-1}\|\vec c\|_2^2
\le Ch^{-d-1}\|\vec c\|_2^2.
\end{align*}
Consequently,
$\lambda_{\max}(K_{h,\varepsilon})
\le Ch^{-d-1}$.
Dividing $\lambda_{\max}$ by
$\lambda_{\min}$ proves \eqref{eq:condition-number}.
\end{proof}

For fixed $\varepsilon\in(0,1]$,
Theorem~\ref{thm:condition-number} gives
$\kappa_2(K_{h,\varepsilon})
\le C_\varepsilon h^{-(4m-1)}$.
For a compactly supported reference kernel of unit support,
the choice $\varepsilon=(C_Wh)^{-1}$ gives support radius $C_Wh$.
Substitution into \eqref{eq:condition-number} yields
$\kappa_2(K_{h,(C_Wh)^{-1}})
\le C C_W^{4m-2d}h^{-(2d-1)}$
on levels satisfying $C_Wh\ge1$.
For fixed $C_W$, an asymptotic analysis as $h\to0$
would also require estimates for $\varepsilon>1$.
Subsection~\ref{subsec:example3} examines how this support
scaling affects accuracy, sparsity, and conditioning
over a finite refinement range.

To compare the conditioning of the kernel Galerkin system
with that of its finite element counterpart, we consider
the $L^2$ mass matrix $M_{h,\varepsilon}$.
Define
$(M_{h,\varepsilon})_{ij}
:=
(\boldsymbol\varphi_{j,\varepsilon},
 \boldsymbol\varphi_{i,\varepsilon})_
 {\mathbf L^2(\boldsymbol\Omega)}$ with the product kernel-translate basis $\{\boldsymbol\varphi_{j,\varepsilon}\}_{j=1}^N$.
This symmetric positive definite Gram matrix represents
the $L^2$ inner product on the trial space, allowing the
coercivity and inverse estimates to yield spectral bounds.
Multiplying $K_{h,\varepsilon}$ on the left and right by
$M_{h,\varepsilon}^{-1/2}$ gives the stiffness matrix in an
$L^2$-orthonormal basis.
The following corollary bounds the condition number of
this symmetric matrix.

\begin{corollary}[Conditioning in an $L^2$-orthonormal basis]
\label{cor:mass-preconditioning}
Under the assumptions of Theorem~\ref{thm:condition-number},
we have
\begin{equation}
\kappa_2\!\left(
M_{h,\varepsilon}^{-1/2}
K_{h,\varepsilon}
M_{h,\varepsilon}^{-1/2}
\right)
\le C\left(q^{-2}+h^{-1}q^{-1}\right)
\le Ch^{-2},
\qquad 0<\varepsilon\le1,
\label{eq:mass-preconditioned-condition-number}
\end{equation}
where $C$ is independent of $h$ and $\varepsilon$.
\end{corollary}

\begin{proof}
To prove \eqref{eq:mass-preconditioned-condition-number},
we use the Rayleigh quotient characterization of the extreme
generalized eigenvalues of $(K_{h,\varepsilon},M_{h,\varepsilon})$.
For every nonzero trial function
$\mathbf v_h\in\mathbf V_{h,\varepsilon}$
with coefficient vector $\vec c$, we need to establish
\begin{equation}
C_*
\le
\frac{\vec c^\top K_{h,\varepsilon}\vec c}
     {\vec c^\top M_{h,\varepsilon}\vec c}
=
\frac{a_h(\mathbf v_h,\mathbf v_h)}
     {\|\mathbf v_h\|_{\mathbf L^2(\boldsymbol\Omega)}^2}
\le C^*\left(q^{-2}+h^{-1}q^{-1}\right),
\label{eq:mass-rayleigh-bounds}
\end{equation}
where $C_*,C^*>0$ are independent of $h$ and $\varepsilon$.
The lower bound follows from coercivity
\eqref{eq:coercivity}, the interface Poincar\'e inequality
\eqref{eq:interface-poincare}, and $h\le1$.
For the upper bound, discrete continuity
\eqref{eq:discrete-continuity}, the Nitsche energy norm
\eqref{eq:nitsche-energy-norm}, and the trace
inequality \cite[Theorem~1.6.6]{BrennerScott2008} give
\begin{equation}
\begin{aligned}
a_h(\mathbf v_h,\mathbf v_h)
&\le C_b\|\mathbf v_h\|_h^2
\le C\left(
\|\mathbf v_h\|_{\mathbf H^1(\boldsymbol\Omega)}^2
+h^{-1}
 \|\mathbf v_h\|_{\mathbf L^2(\boldsymbol\Omega)}
 \|\mathbf v_h\|_{\mathbf H^1(\boldsymbol\Omega)}
\right),
\end{aligned}
\label{eq:mass-energy-inverse}
\end{equation}
where the jump is bounded by the two one-sided traces.
Thus it suffices to prove
$\|\mathbf v_h\|_{\mathbf H^1(\boldsymbol\Omega)}
\le Cq^{-1}
\|\mathbf v_h\|_{\mathbf L^2(\boldsymbol\Omega)}$.
Sobolev interpolation and the scale-uniform Bernstein estimate
\eqref{eq:scaled-bernstein} with $\alpha=2$,
applied componentwise, give by Cauchy--Schwarz
\begin{align*}
\|\mathbf v_h\|_{\mathbf H^1(\boldsymbol\Omega)}^2
&\le
C\|\mathbf v_h\|_{\mathbf L^2(\boldsymbol\Omega)}
 \|\mathbf v_h\|_{\mathbf H^2(\boldsymbol\Omega)}
\le
Cq^{-1}\|\mathbf v_h\|_{\mathbf L^2(\boldsymbol\Omega)}
 \|\mathbf v_h\|_{\mathbf H^1(\boldsymbol\Omega)}.
\end{align*}
After division by the nonzero $H^1$ norm, substitution into
the energy estimate \eqref{eq:mass-energy-inverse}
proves the upper bound in \eqref{eq:mass-rayleigh-bounds}.
These Rayleigh quotient bounds and
$q\ge(c_h/\rho_0)h$ from the center assumptions
\eqref{eq:center-assumptions} establish
\eqref{eq:mass-preconditioned-condition-number}.
\end{proof}

The numerical stability of constructing this basis in finite
precision requires separate analysis.
Subsection~\ref{subsec:example3} compares condition estimates
in the kernel-translate and $L^2$-orthonormal bases.

\section{Numerical experiments}
\label{sec:numerics}
We first use two planar interface problems to examine fixed-scale
convergence, and support-radius
effects. A three-dimensional heat-conduction example then illustrates
the implementation on a curved composite with mixed boundary conditions.
All discretizations use the symmetric Nitsche formulation
\eqref{eq:discrete-nitsche} with
$\gamma_\Gamma=\gamma_{\partial,+}=100$.
Independent kernel spaces are constructed on the physical subdomains,
and the integrals are evaluated by numerical quadrature.
Errors are measured in $\mathbf L^2(\boldsymbol\Omega)$ and the
Nitsche energy norm \eqref{eq:nitsche-energy-norm}.
Kernel parameters, quadrature rules, and the solutions used for
comparison are specified below.

\medskip
\noindent\textbf{Problem A (smooth manufactured solution).}
Let $\Omega$ and $\Omega^-$ be the interiors of
$r=\rho_B(\theta)$ and $r=\rho_\Gamma(\theta)$, respectively, where
\begin{equation}
\begin{aligned}
\rho_\Gamma(\theta)
 &=0.6+0.216\sin(3\theta)+0.096\cos(2\theta)+0.24\cos(5\theta),\\
\rho_B(\theta)
 &=1.5+0.14\sin(4\theta)+0.12\cos(6\theta)+0.09\cos(5\theta).
\end{aligned}
\label{eq:ex1-curved-geometry}
\end{equation}
Set $\Omega^+=\Omega\setminus\overline{\Omega^-}$. In the product notation of Section~\ref{sec:setting}, prescribe
\begin{equation}
\begin{aligned}
\mathbf A=\bigl((1+x^2+y^2)I_2,\,(2+xy)I_2\bigr), \qquad
\mathbf u=\bigl(\cos(x+y),\,\sin(x+y)\bigr).
\end{aligned}
\label{eq:ex1-curved-solution}
\end{equation}
The source, interface, and boundary data follow by substitution into
\eqref{eq:intro-interface-problem}.

\medskip
\noindent\textbf{Problem B (numerical reference solution).}
Let $\Omega=(-1,1)^2$, $\Omega^-=\{(x,y):|x|+|y|<1/2\}$, and
$\Omega^+=\Omega\setminus\overline{\Omega^-}$.
We prescribe $\mathbf A=(I_2,10I_2)$, $\mathbf f=(1,1)$,
$g_D=0$, $g_\Gamma=\sqrt2\,y$, and
$\psi_\Gamma=\sqrt2\,(\vec n_\Gamma)_y$.
No exact solution is prescribed. 
Using the corner analysis \cite{Nicaise1990} we can establish finite piecewise
regularity and gives
\begin{equation}
u^\eta\in H^{1+\lambda_\star-\tau}(\Omega^\eta),
\qquad \lambda_\star\approx1.268308,\qquad \tau>0, \qquad \eta=\pm.
\label{eq:ex2-regularity}
\end{equation}
A fitted $P^2$ Nitsche finite element solution with mesh size
$h_{\rm FE}=1/256$ is used as the reference for Problem~B in
Examples~2 and~3. Its assembly uses 3-point Gauss rules on
edges and in each Duffy coordinate.
The solution is computed once and reused, and all reported errors
for Problem~B are differences from this fixed reference solution.

For both problems, independent shifted triangular lattices are clipped
to the two subdomains; no centers are added on the boundaries or
interface. Fig.~\ref{fig:test-geometries} shows representative center
sets and the reference solution. We write $N=|X^-|+|X^+|$ and
$h=\max\{h_-,h_+\}$. Errors are evaluated against the exact solution for Problem~A
and the fixed reference solution for Problem~B.

\begin{figure}[t]
\centering
\begin{subfigure}[t]{0.29\textwidth}
    \centering
    \includegraphics[width=\textwidth]
    {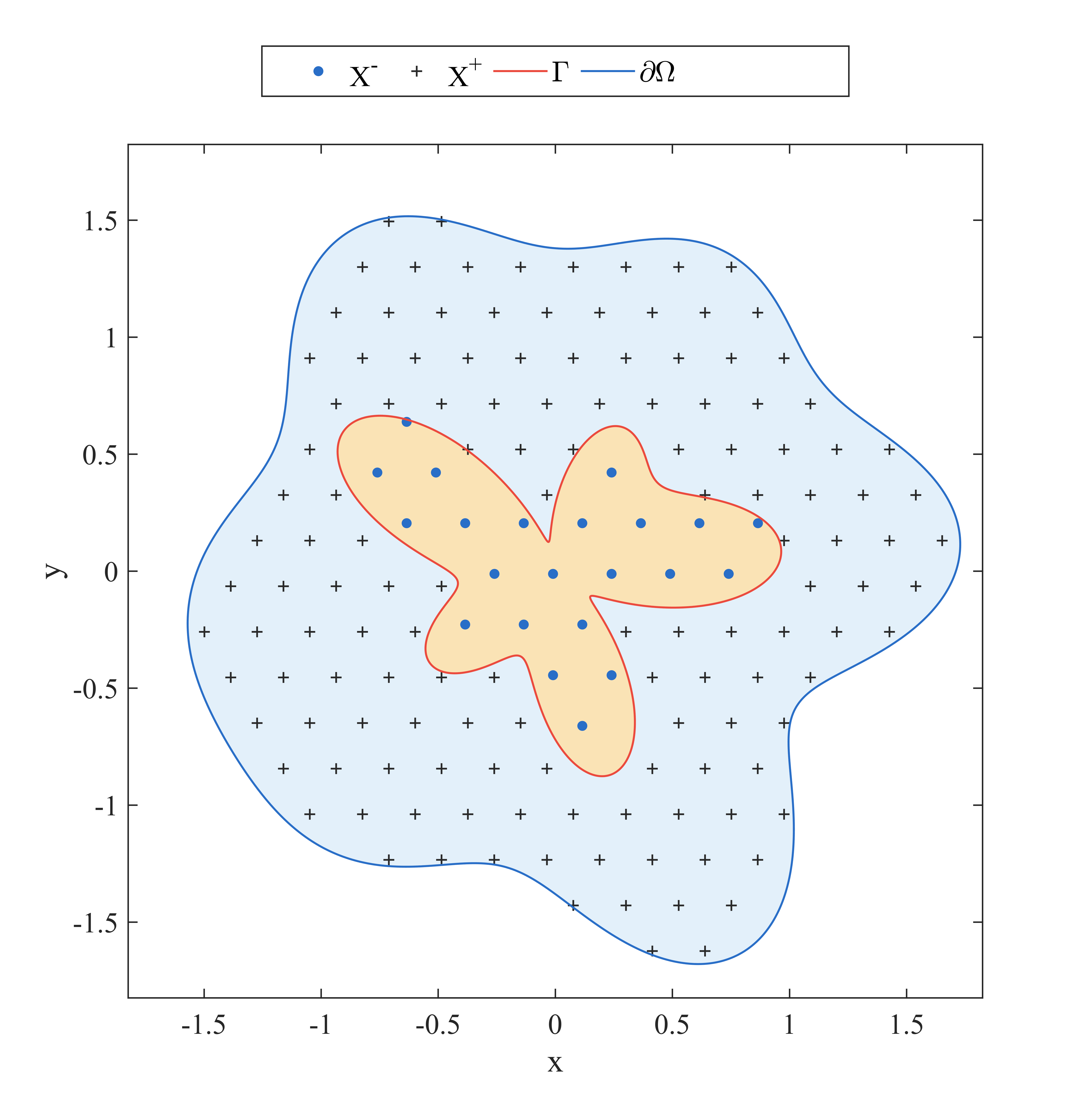}
    \caption{Trial centers for Problem~A.}
\end{subfigure}
\hfill
\begin{subfigure}[t]{0.29\textwidth}
    \centering
    \includegraphics[width=\textwidth]
    {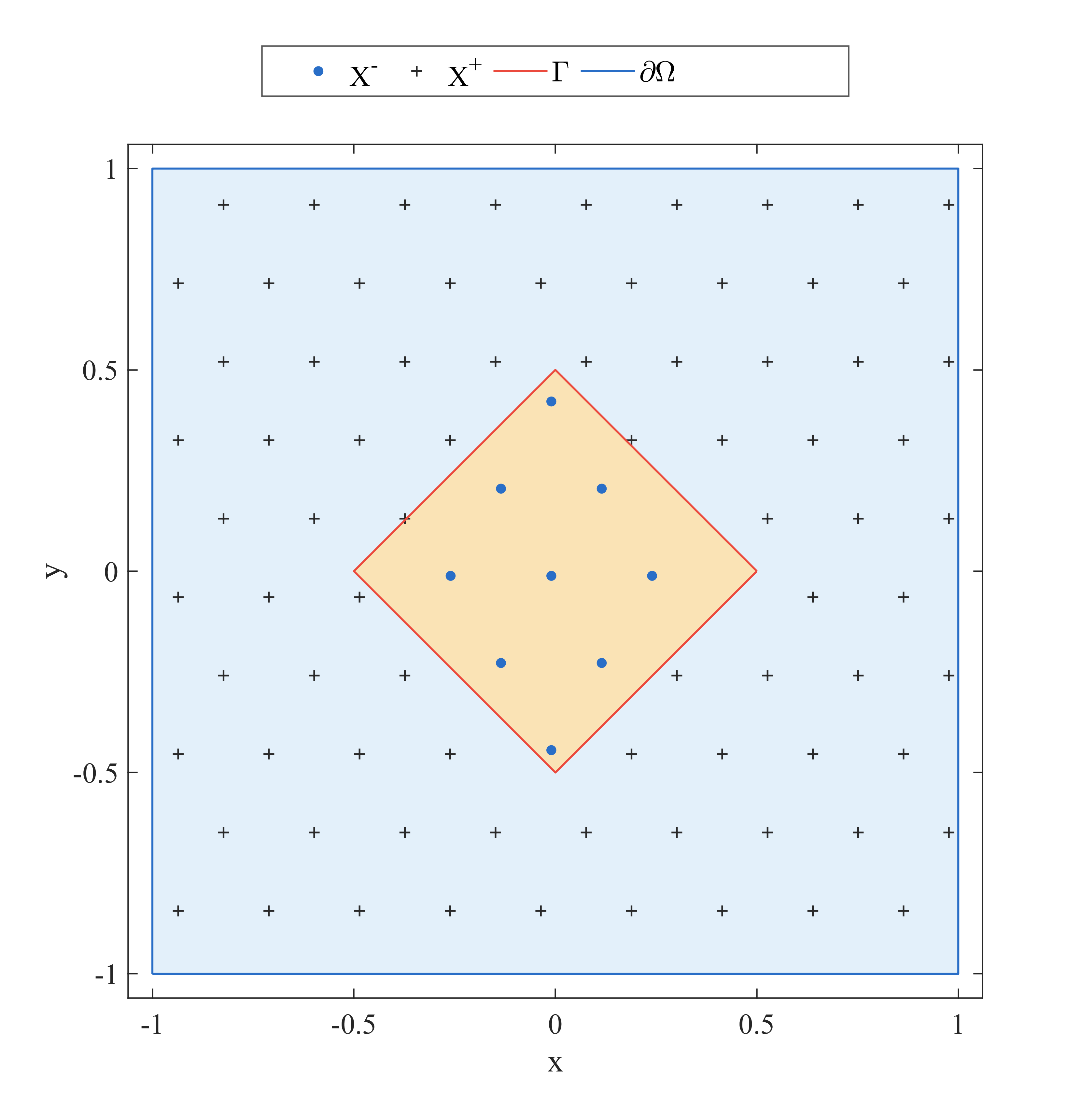}
    \caption{Trial centers for Problem~B.}
\end{subfigure}
\hfill
\begin{subfigure}[t]{0.33\textwidth}
    \centering
    \includegraphics[width=\textwidth]
    {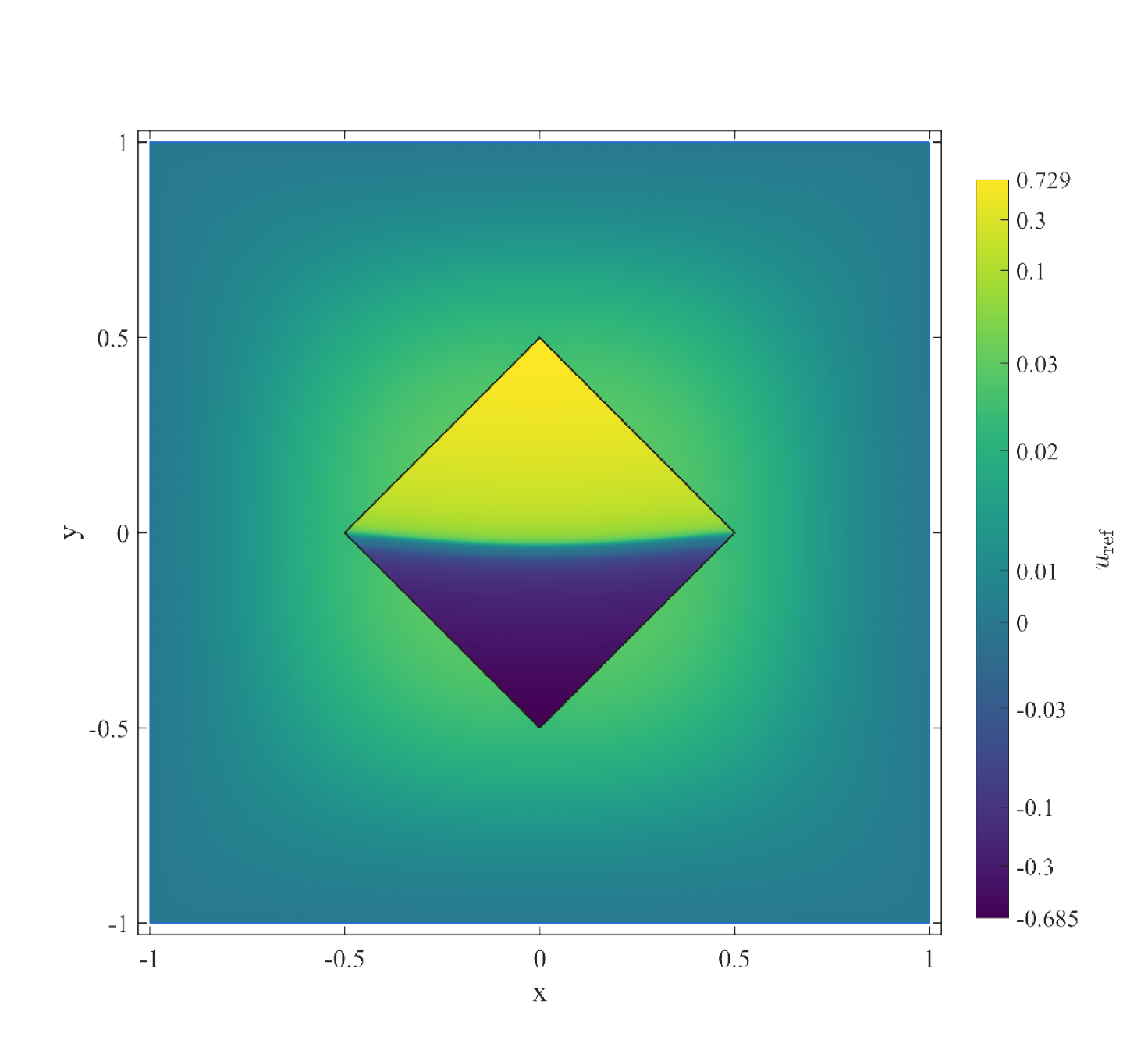}
    \caption{Reference solution for Problem~B.}
\end{subfigure}
\vspace{-4mm}
\caption{Representative trial center distributions and the fitted
$P^2$ Nitsche finite element reference solution.}
\label{fig:test-geometries}
\vspace{-4mm}
\end{figure}

\subsection{Convergence for a smooth manufactured solution}
\label{subsec:example1}
We test the fixed-scale convergence estimates for Problem~A
using Mat\'ern and Wendland kernels with matching Sobolev orders
$m=2.5,3.5$.
The scale parameters are $\varepsilon=1,3,5$ for the Mat\'ern
kernels and $\varepsilon=1/4,1/2,1$ for the Wendland kernels.
The center sets contain
$N=77,156,318,631,1259,2517$ points, with measured fill
distances decreasing from $0.34835$ to $0.06190$.
All kernel settings use the same centers at each level,
and each scale is fixed under refinement.
Assembly uses 10-point Gaussian rules in each coordinate
on $(12,16,129)$ inner-radial, outer-radial, and angular panels.
Error evaluation uses 12-point rules and twice as many panels.
The computations use double precision.

Fig.~\ref{fig:ex1-L2-energy} compares the $L^2$ and Nitsche
energy errors with the reference rates $h^m$ and $h^{m-1}$
from \eqref{eq:L2-error} and \eqref{eq:energy-error}.
The matching Sobolev orders give the same reference rates
for both kernel families.
For $m=2.5$, both errors decrease throughout the tested
refinement range at every scale.
For $m=3.5$, the Mat\'ern kernels with $\varepsilon=3,5$
and the Wendland kernel with $\varepsilon=1$ retain decreasing
errors throughout this range.
In contrast, the Mat\'ern kernel with $\varepsilon=1$
and the Wendland kernel with $\varepsilon=1/4$
exhibit nonmonotone behavior after an initial reduction,
while the Wendland errors with $\varepsilon=1/2$
increase at the finest level.
The regularly decreasing curves exhibit decay consistent
with the reference orders over resolved portions of the
refinement range.
Increasing the Sobolev order therefore does not improve
accuracy uniformly over the tested scales and refinement
levels in double precision.
\begin{figure}[t]
\centering
    \includegraphics[width=0.9\textwidth]
    {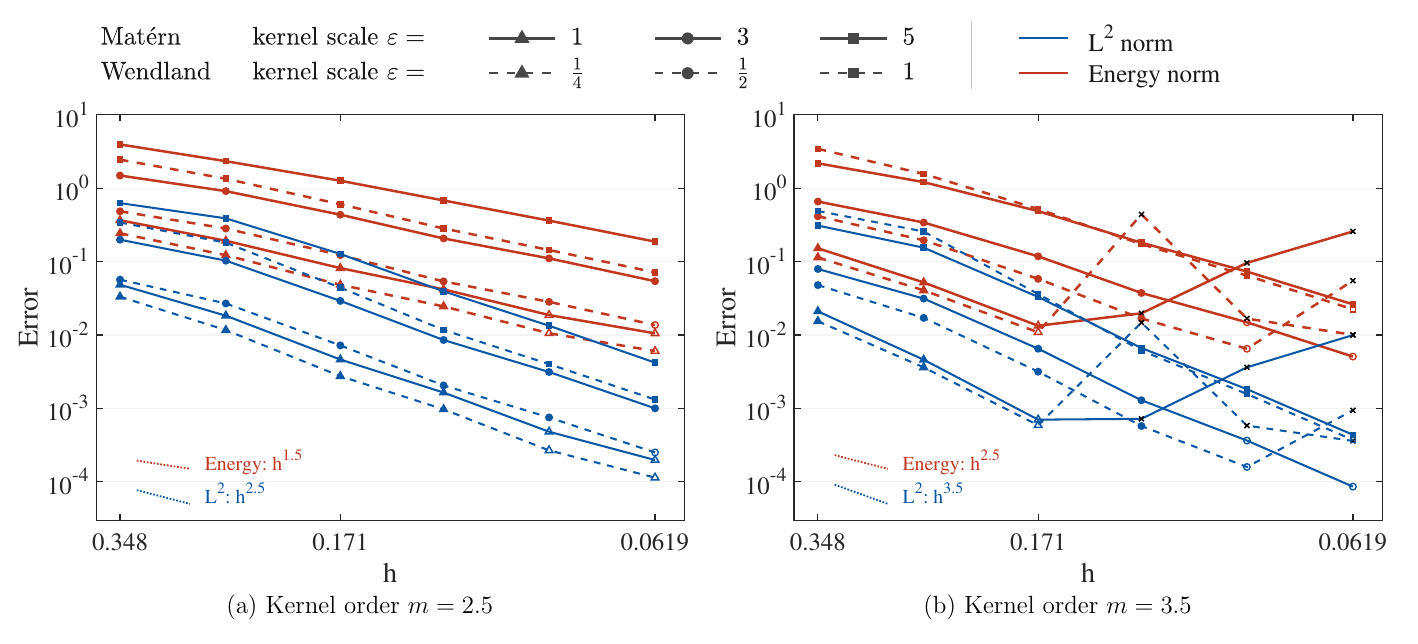}
\vspace{-4mm}
\caption{
(Ex1)
Fixed-scale convergence for the smooth solution of Problem A using Mat\'ern and Wendland kernels with kernel scale $\varepsilon$. Errors are plotted against the fill distance $h = \max\{h_-, h_+\}$. Dotted reference lines show the theoretical convergence rates proved in the preliminaries: $\mathcal{O}(h^{m-1})$ for the energy norm and $\mathcal{O}(h^m)$ for the $L^2$ norm.}
\label{fig:ex1-L2-energy}
\vspace{-4mm}
\end{figure}

To examine the spatial distribution of the error,
Fig.~\ref{fig:ex1-error-distribution} shows the
absolute errors for $m=3$ and $\varepsilon=5$ on the three
finest center sets.
Larger errors occur near portions of the interface and
exterior boundary, while the subdomain interiors generally
have smaller errors.
\begin{figure}[t]
\centering
\begin{subfigure}[t]{0.32\textwidth}
    \centering
    \includegraphics[width=\textwidth]
    {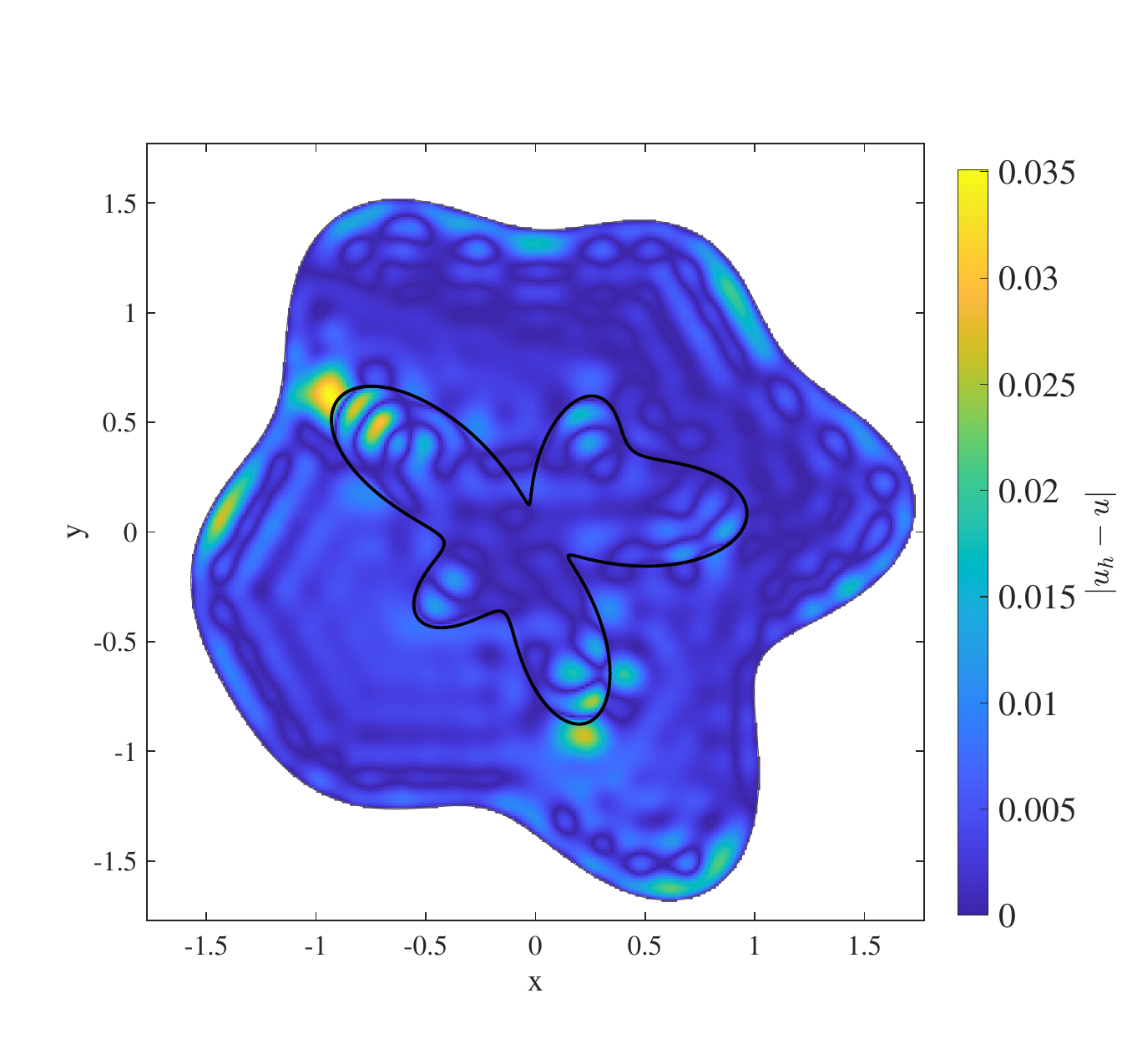}
    \caption{$N=631$.}
    \label{fig:ex1-error-distribution-N631}
\end{subfigure}
\hfill
\begin{subfigure}[t]{0.32\textwidth}
    \centering
    \includegraphics[width=\textwidth]
    {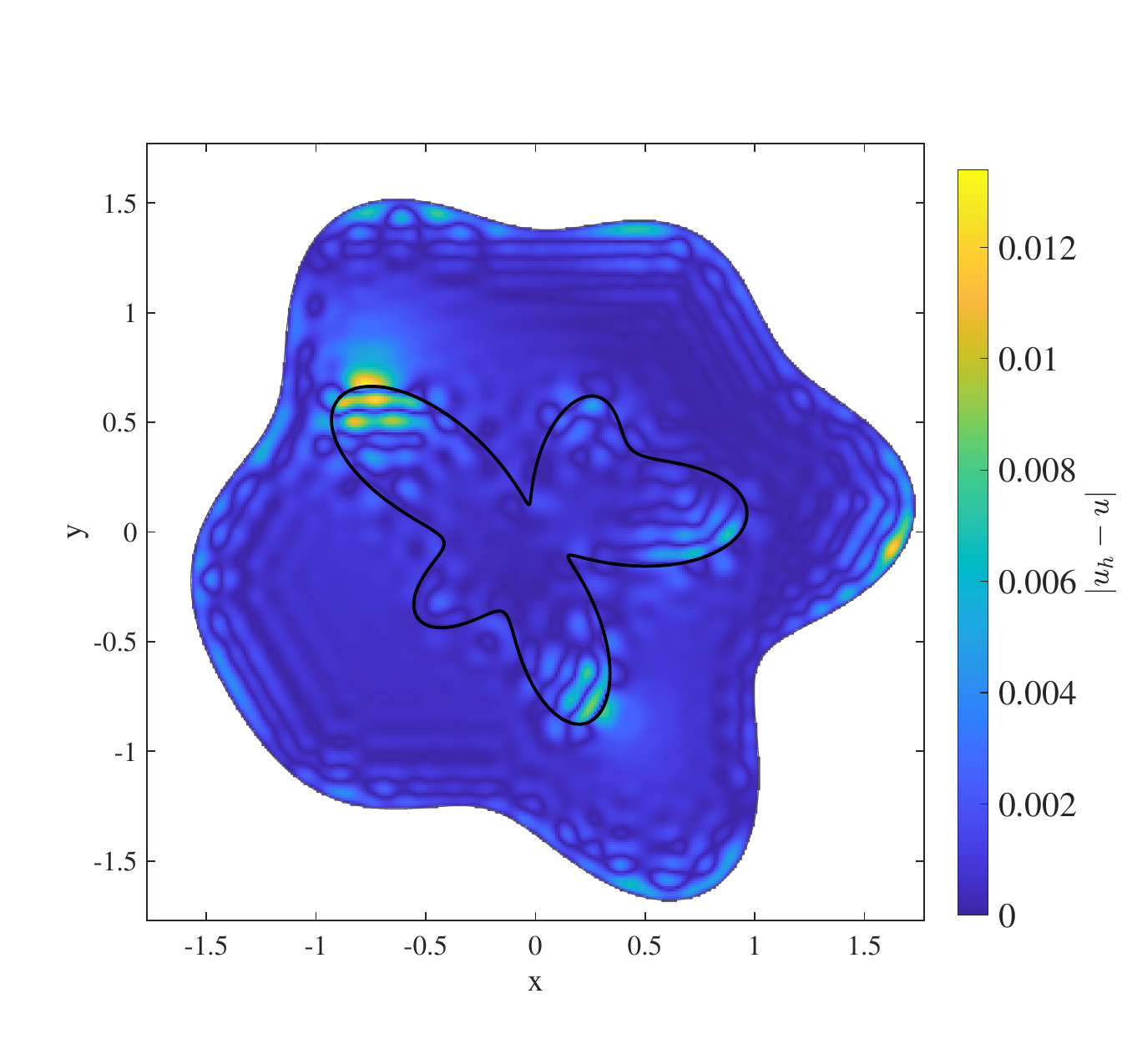}
    \caption{$N=1259$.}
    \label{fig:ex1-error-distribution-N1259}
\end{subfigure}
\hfill
\begin{subfigure}[t]{0.32\textwidth}
    \centering
    \includegraphics[width=\textwidth]
    {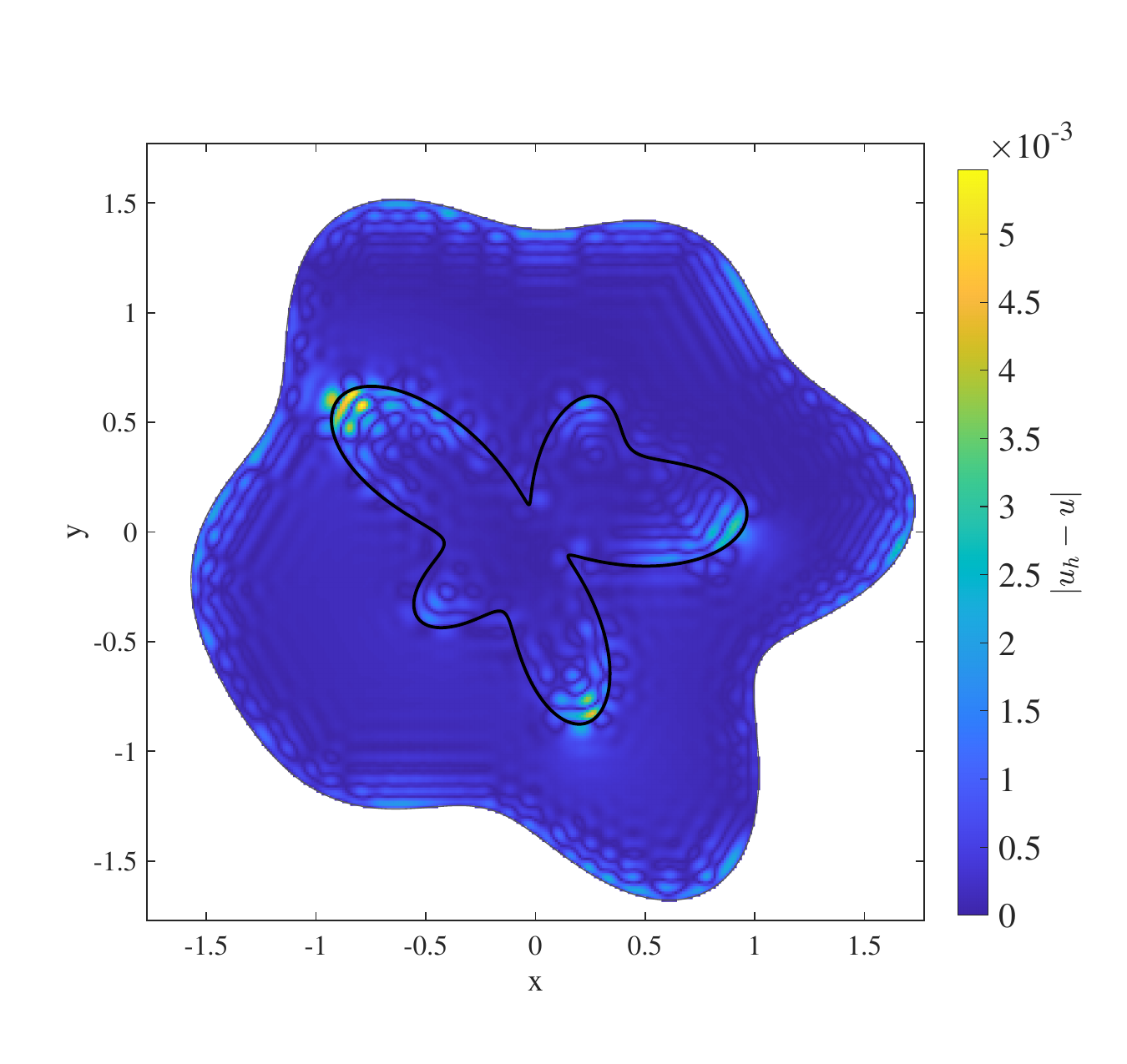}
    \caption{$N=2517$.}
    \label{fig:ex1-error-distribution-N2517}
\end{subfigure}
\vspace{-4mm}
\caption{
(Ex1) Error profiles
for Problem~A with Mat\'ern $m=3$ and $\varepsilon=5$.
Each panel uses a separate color scale to resolve the
spatial error distribution.
The black curve marks the interface $\Gamma$.}
\label{fig:ex1-error-distribution}
\vspace{-4mm}
\end{figure}

\subsection{Convergence for a numerical reference solution of limited regularity}
\label{subsec:example2}
We refine the kernel centers for Problem~B using Mat\'ern
kernels with order $m=2,3,4$ and kernel scale $\varepsilon=5$, and Wendland
kernels with $m=2.5,~3.5$ and $\varepsilon=\frac 12$.
The five center sets have measured fill distances decreasing
from $0.404$ to $0.0964$.
Kernel assembly uses a separate interface-fitted integration
mesh based on $32$ subdivisions per coordinate direction,
with 10-point Gaussian rules in each Duffy coordinate.
Errors relative to the saved finite element reference solution
are evaluated on the reference mesh using 12-point rules.

Table~\ref{tab:ex2-errors} reports the $L^2$ and Nitsche energy
errors against the measured fill distance $h$.
Within each family, higher Sobolev orders give smaller errors
on the finer center sets.
For Problem~B, \eqref{eq:energy-error} and
\eqref{eq:ex2-regularity} give energy-error order $1$
for Mat\'ern $m=2$ and $\lambda_\star-\tau$ for the other
tested kernels, with sufficiently small $\tau>0$.
The fitted energy-error exponent for Mat\'ern $m=2$ is close
to $1$, whereas the remaining exponents exceed these
regularity-based orders.
This faster decay may reflect a preasymptotic regime
in which the corner singularities do not yet dominate
the total error.
\begin{table}[t]
\centering
\caption{
(Ex2) Problem B errors relative to a reference finite element solution. Kernels $M_m$ and $W_m$ denote Mat\'{e}rn and Wendland kernels of Sobolev order $m$ with scale parameters $\varepsilon = 5$ and $\varepsilon = 1/2$, respectively. Measured convergence rates ($\mathit{Rate}$) are compared against theoretical benchmark rates ($\mathit{Rate}^*$).
}
\label{tab:ex2-errors}

\begingroup
\footnotesize
\setlength{\tabcolsep}{2.1pt}
\setlength{\arrayrulewidth}{0.6pt}
\renewcommand{\arraystretch}{1.18}

\begin{tabular}{@{}r rrrrr|rrrrr@{}}
\toprule
\multicolumn{1}{c}{}
& \multicolumn{5}{c|}{Nitsche energy error}
& \multicolumn{5}{c}{$L^2$ error}
\\
\cmidrule(lr){2-6}\cmidrule(l){7-11}
$h$
& $\mathrm M_2$ & $\mathrm M_3$ & $\mathrm M_4$
& $\mathrm W_{5/2}$ & $\mathrm W_{7/2}$
& $\mathrm M_2$ & $\mathrm M_3$ & $\mathrm M_4$
& $\mathrm W_{5/2}$ & $\mathrm W_{7/2}$
\\
\midrule

 0.40357
& $1.97\mathrm{e}{+0}$
& $1.32\mathrm{e}{+0}$
& $9.69\mathrm{e}{-1}$
& $4.42\mathrm{e}{-1}$
& $4.30\mathrm{e}{-1}$
& $2.65\mathrm{e}{-1}$
& $1.67\mathrm{e}{-1}$
& $1.15\mathrm{e}{-1}$
& $5.67\mathrm{e}{-2}$
& $5.41\mathrm{e}{-2}$
\\

0.26025
& $8.75\mathrm{e}{-1}$
& $5.91\mathrm{e}{-1}$
& $4.43\mathrm{e}{-1}$
& $1.81\mathrm{e}{-1}$
& $2.07\mathrm{e}{-1}$
& $7.56\mathrm{e}{-2}$
& $5.27\mathrm{e}{-2}$
& $4.01\mathrm{e}{-2}$
& $1.65\mathrm{e}{-2}$
& $1.97\mathrm{e}{-2}$
\\

 0.18068
& $8.67\mathrm{e}{-1}$
& $4.12\mathrm{e}{-1}$
& $2.17\mathrm{e}{-1}$
& $1.25\mathrm{e}{-1}$
& $7.01\mathrm{e}{-2}$
& $5.80\mathrm{e}{-2}$
& $2.31\mathrm{e}{-2}$
& $1.29\mathrm{e}{-2}$
& $8.07\mathrm{e}{-3}$
& $4.87\mathrm{e}{-3}$
\\

 0.13549
& $5.48\mathrm{e}{-1}$
& $2.07\mathrm{e}{-1}$
& $8.28\mathrm{e}{-2}$
& $6.09\mathrm{e}{-2}$
& $2.57\mathrm{e}{-2}$
& $2.63\mathrm{e}{-2}$
& $7.94\mathrm{e}{-3}$
& $3.17\mathrm{e}{-3}$
& $2.68\mathrm{e}{-3}$
& $1.09\mathrm{e}{-3}$
\\

 0.09637
& $4.46\mathrm{e}{-1}$
& $1.44\mathrm{e}{-1}$
& $4.08\mathrm{e}{-2}$
& $4.50\mathrm{e}{-2}$
& $8.77\mathrm{e}{-3}$
& $1.46\mathrm{e}{-2}$
& $4.60\mathrm{e}{-3}$
& $1.27\mathrm{e}{-3}$
& $1.64\mathrm{e}{-3}$
& $2.61\mathrm{e}{-4}$
\\

\midrule
\multicolumn{1}{l}{Rate}
& 1.04 & 1.65 & 2.64 & 1.61 & 3.30
& 2.18 & 2.54 & 3.65 & 2.50 & 4.64
\\
\multicolumn{1}{l}{Rate*}
& $1$ & $2$ & $3$ & $1.5$ & $2.5$
& $2$ & $3$ & $4$ & $2.5$ & $3.5$
\\
\bottomrule
\end{tabular}
\endgroup
\end{table}

\subsection{Accuracy, sparsity, and conditioning under stationary scaling}
\label{subsec:example3}
The preceding experiments use fixed kernel scales.
We now adopt stationary scaling, which keeps the support radius
proportional to the fill distance
\cite[Sections~7.4--7.5]{SchabackWendland2006}.
For the Wendland kernel of Sobolev order $m=3.5$, we set
$\varepsilon=(C_Wh)^{-1}$ with $C_W\in\{5,10,15,20\}$.
Thus $\varepsilon h=C_W^{-1}$ remains fixed under refinement, and $C_W$
controls the support overlap.
In two dimensions, substitution into \eqref{eq:condition-number}
gives
$\kappa_2(K_{h,(C_Wh)^{-1}})
\le C C_W^{10}h^{-3}$ for $C_Wh\ge1$.
This restriction is inherited from $\varepsilon\le1$.
All reported levels with $C_W\ge10$ satisfy it.
For $C_W=5$, the two finest levels of both problems also test
scales outside this range.
The power $h^{-3}$ provides a reference for the conditioning
comparison within the stated range.

We use the first four center sets from Examples~1 and~2,
with the same penalty parameters and quadrature constructions.
At each level, all support choices use identical centers and
common integration rules refined to resolve the smallest support.
The $L^2$ errors use the exact solution for Problem~A and the saved
finite element reference solution for Problem~B.
Fig.~\ref{fig:ex3-scaling} compares these errors with the matrix
nonzero counts and condition estimates in the kernel-translate basis.
Over the reported refinement range, increasing $C_W$ reduces
the $L^2$ errors and improves their observed convergence rates.
The fitted error exponents remain below the fixed-scale benchmark
$m=3.5$ available for sufficiently smooth solutions.
The condition estimates have fitted growth exponents below three.
Thus, stationary scaling with a suitable $C_W$ combines
observable error convergence with moderate growth of coefficient
conditioning over these levels.
Larger support overlap nevertheless increases the absolute
condition estimates and matrix density. The condition estimates in the $L^2$-orthonormal basis
are substantially smaller and show little variation
with $C_W$ over the reported refinement levels.
The $h^{-2}$ references show the growth power in
Corollary~\ref{cor:mass-preconditioning}.
\begin{figure}[t]
\centering

\begin{minipage}[c]{0.035\textwidth}
    \centering
    \rotatebox[origin=c]{90}{\small Problem~A}
\end{minipage}
\hfill
\begin{subfigure}[c]{0.31\textwidth}
    \centering
    \includegraphics[width=\textwidth]{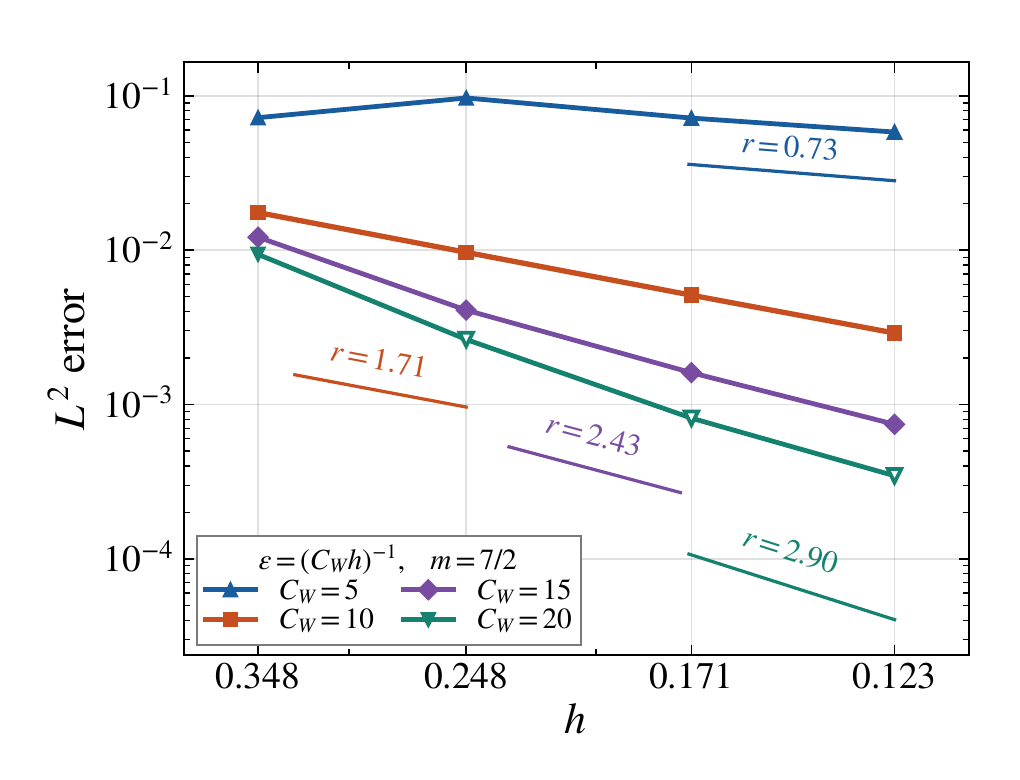}
\end{subfigure}
\hfill
\begin{subfigure}[c]{0.31\textwidth}
    \centering
    \includegraphics[width=\textwidth]{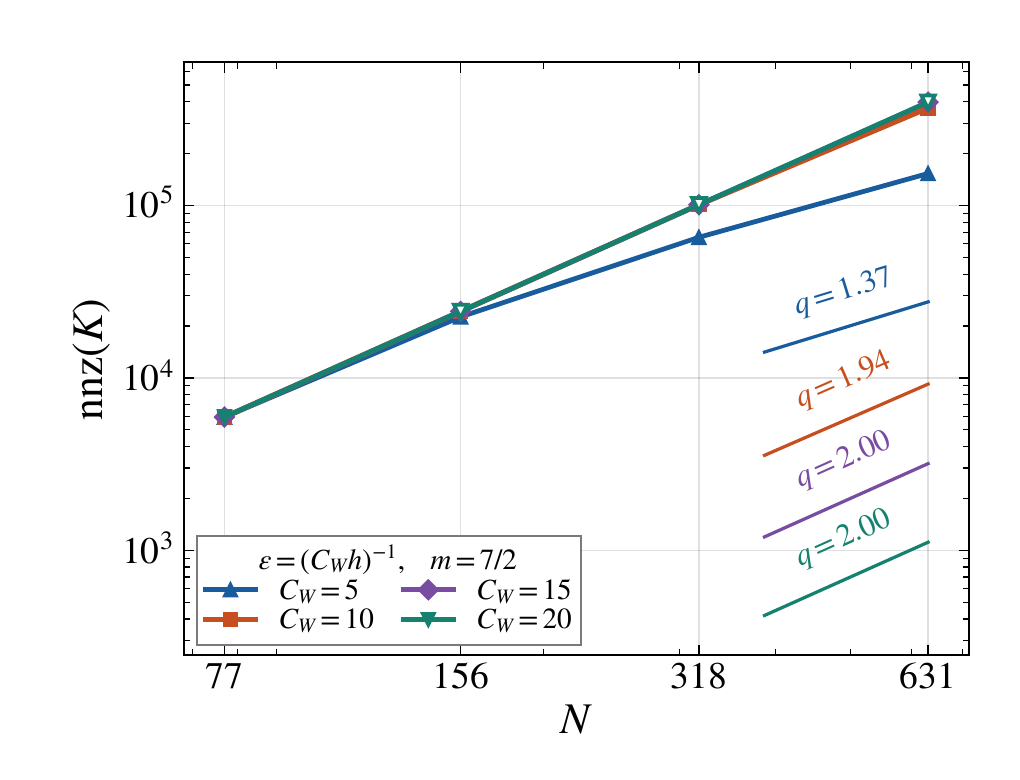}
\end{subfigure}
\hfill
\begin{subfigure}[c]{0.31\textwidth}
    \centering
    \includegraphics[width=\textwidth]{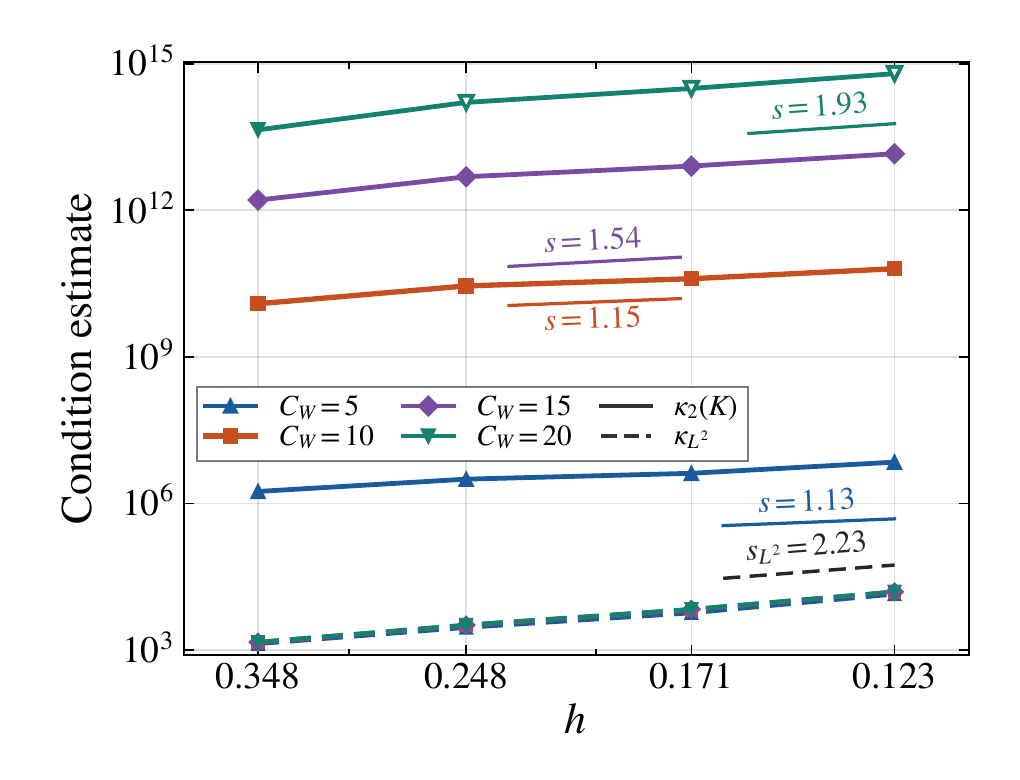}
\end{subfigure}

\par\vspace{-1mm}

\begin{minipage}[c]{0.035\textwidth}
    \centering
    \rotatebox[origin=c]{90}{\small Problem~B}
\end{minipage}
\hfill
\begin{subfigure}[c]{0.31\textwidth}
    \centering
    \includegraphics[width=\textwidth]{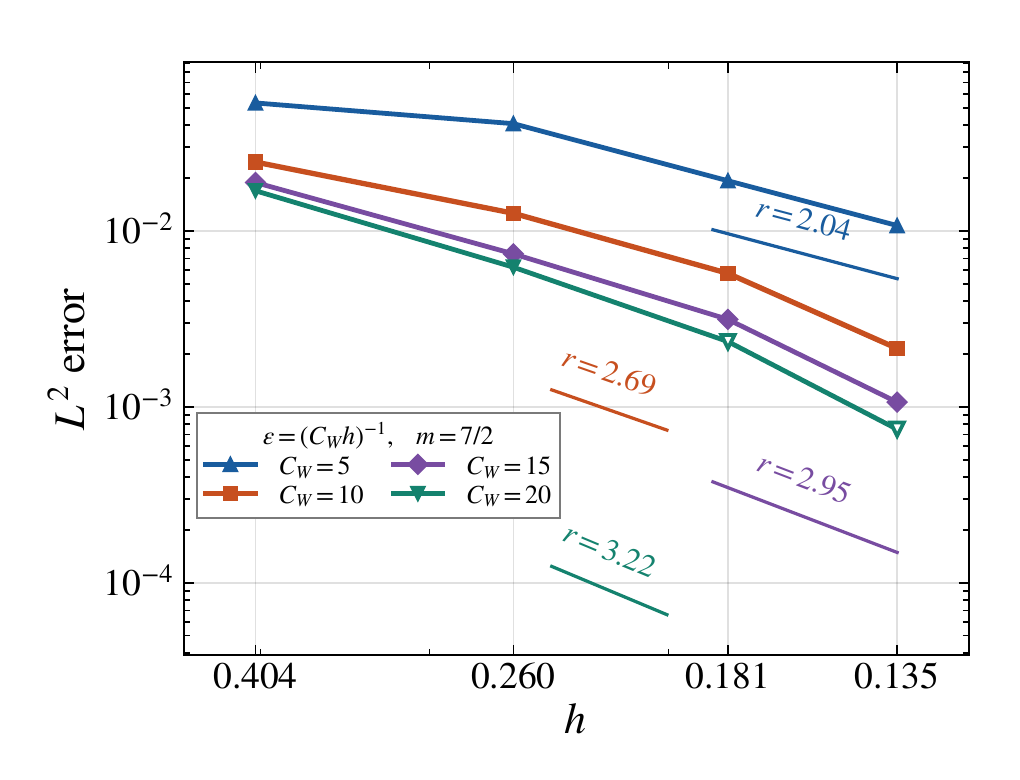}
\end{subfigure}
\hfill
\begin{subfigure}[c]{0.31\textwidth}
    \centering
    \includegraphics[width=\textwidth]{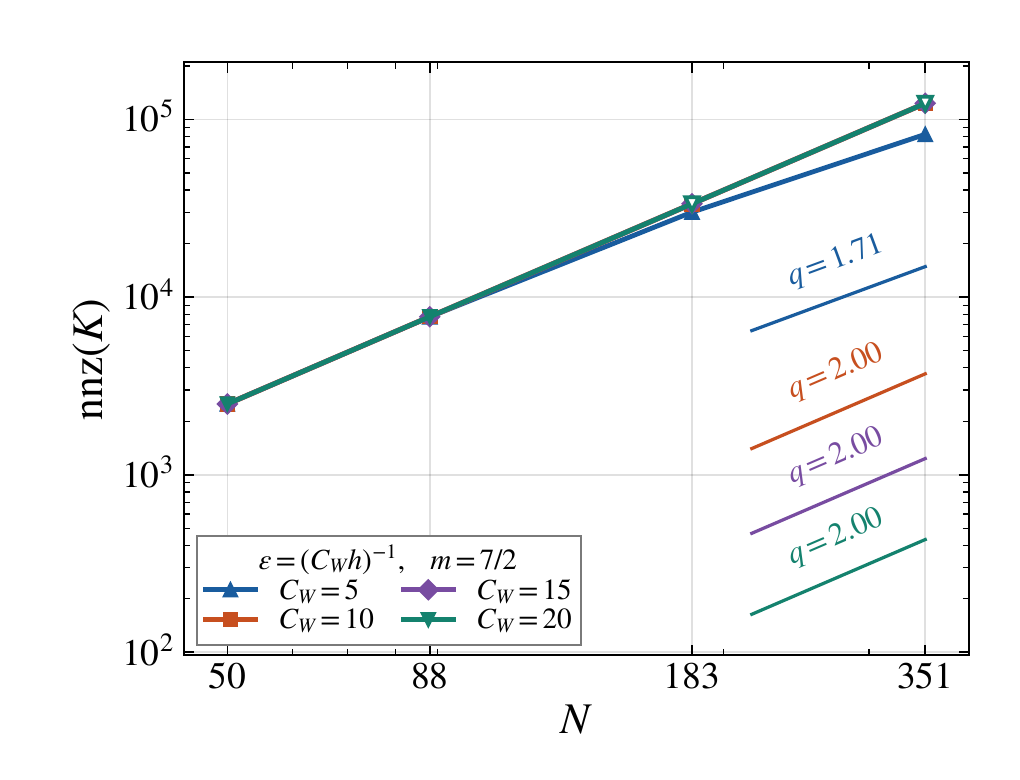}
\end{subfigure}
\hfill
\begin{subfigure}[c]{0.31\textwidth}
    \centering
    \includegraphics[width=\textwidth]{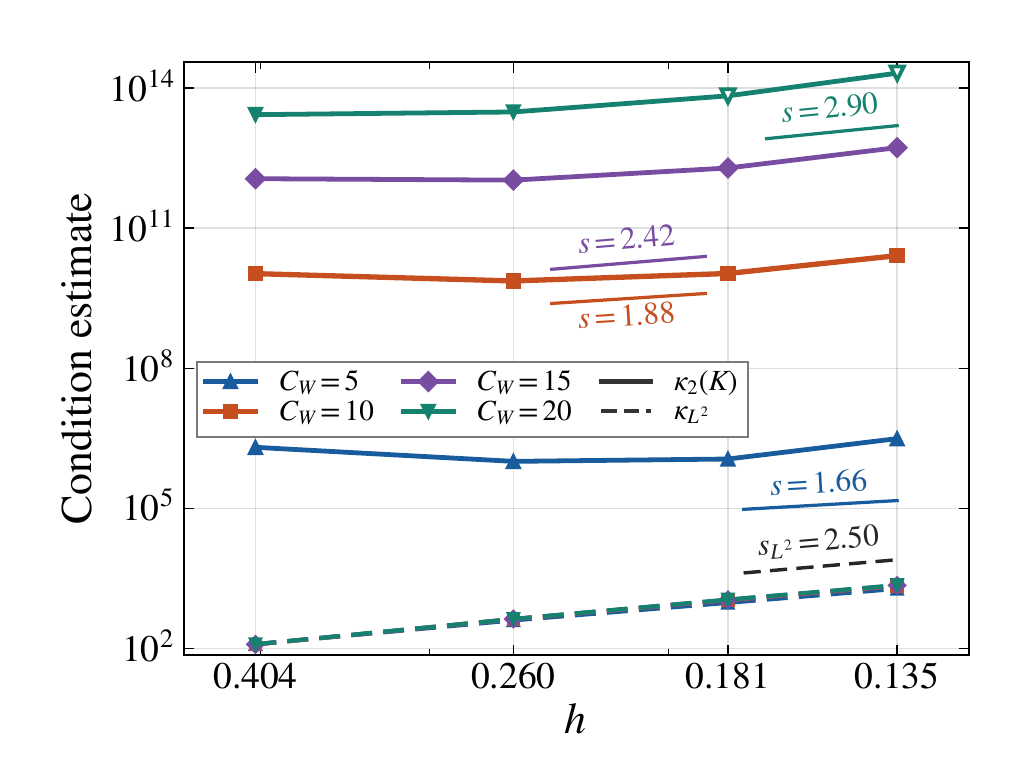}
\end{subfigure}

\par\vspace{-1mm}

\hspace*{0.035\textwidth}%
\hfill
\begin{minipage}[t]{0.31\textwidth}
    \centering
    \small \hspace{8mm}  $L^2$ error
\end{minipage}
\hfill
\begin{minipage}[t]{0.31\textwidth}
    \centering
    \small \hspace{8mm} Matrix nonzero count
\end{minipage}
\hfill
\begin{minipage}[t]{0.31\textwidth}
    \centering
    \small \hspace{8mm} Condition estimate
\end{minipage}
\vspace{-1mm}
\caption{
(Ex3) Accuracy, sparsity, and conditioning under stationary scaling
$\varepsilon=(C_Wh)^{-1}$ for Wendland kernel of Sobolev order $m=3.5$.
The scale is $\varepsilon=(C_Wh)^{-1}$ with
$C_W\in\{5,10,15,20\}$. 
In the right panels, $\kappa_2(K)$ denotes the condition
number in the kernel-translate basis, shown by solid curves.
The dashed curves show
$\kappa_{L^2}:=\kappa_2(M^{-1/2}KM^{-1/2})$,
where $K=K_{h,\varepsilon}$ and $M=M_{h,\varepsilon}$.
Here fill distance is $h=\max\{h_-,h_+\}$ and trial centers have $N=|X^-|+|X^+|$.}
\vspace{-4mm}
\label{fig:ex3-scaling}
\end{figure}

\subsection{Heat conduction in a three-dimensional curved composite}
\label{subsec:example4}
We consider the two-material thermal-shielding tile
of \cite{WassermannEtAl2020}. The geometry is given by
$\mathbf F(\xi,\theta,s)=\mathbf S(\xi,\theta)+s\mathbf e_3$,
where $\mathbf S$ is the quadratic B-spline base surface defined
there, $(\xi,\theta)\in(0,1)^2$, and $0<s<5\,\mathrm{cm}$.
Titanium occupies $0<s<1.25\,\mathrm{cm}$ and porous silica
occupies $1.25<s<5\,\mathrm{cm}$.
Their conductivities are
$\alpha_{\rm Ti}=0.216\,\mathrm{W/(cm\,K)}$ and
$\alpha_{\rm SiO_2}=2.3\times10^{-3}\,\mathrm{W/(cm\,K)}$.
We solve $-\nabla\cdot(\alpha\nabla T)=0$ with
$T=20^\circ\mathrm C$ on $s=0$ and
$T=1000^\circ\mathrm C$ on $s=5\,\mathrm{cm}$.
The remaining exterior surface is insulated.
Perfect thermal contact imposes
$\jump{T}=0$ and
$\jump{\alpha\nabla T\cdot\mathbf n_\Gamma}=0$
at $s=1.25\,\mathrm{cm}$.

We use a fixed-scale Mat\'ern kernel with $m=3$ and
$\varepsilon=0.3\sqrt{2}\,\mathrm{cm}^{-1}$.
The Nitsche penalty parameters are also selected as
$\gamma_\Gamma=\gamma_{\partial,\pm}=100$.
Independent shifted Cartesian lattices are clipped to the
material subdomains.
Integration uses the B-spline parametrization.
We present the calculation at $h=0.6\,\mathrm{cm}$ with
$N=11320$ centers, comprising $2871$ in titanium and $8449$ in silica.

Fig.~\ref{fig:example4-3d} shows the trial centers,
temperature field, and numerical interface jump.
Fig.~\ref{fig:example4-sections} shows the section
$x=5\,\mathrm{cm}$ and the profile $T_h(5,25,z)$ at the same
resolution.
Most of the temperature rise occurs in the poorly conducting
silica layer.
The profile shows a sharp change in slope at the interface.
The numerical temperature jump is small relative to the
imposed temperature difference, consistent with the weak
enforcement of thermal contact. The through-thickness profile is qualitatively consistent with
the discontinuous-material case reported in
\cite[Fig.~16(a)]{WassermannEtAl2020}.
\begin{figure}[t]
\centering
\begin{subfigure}[t]{0.32\textwidth}
    \centering
    \includegraphics[width=\textwidth]
    {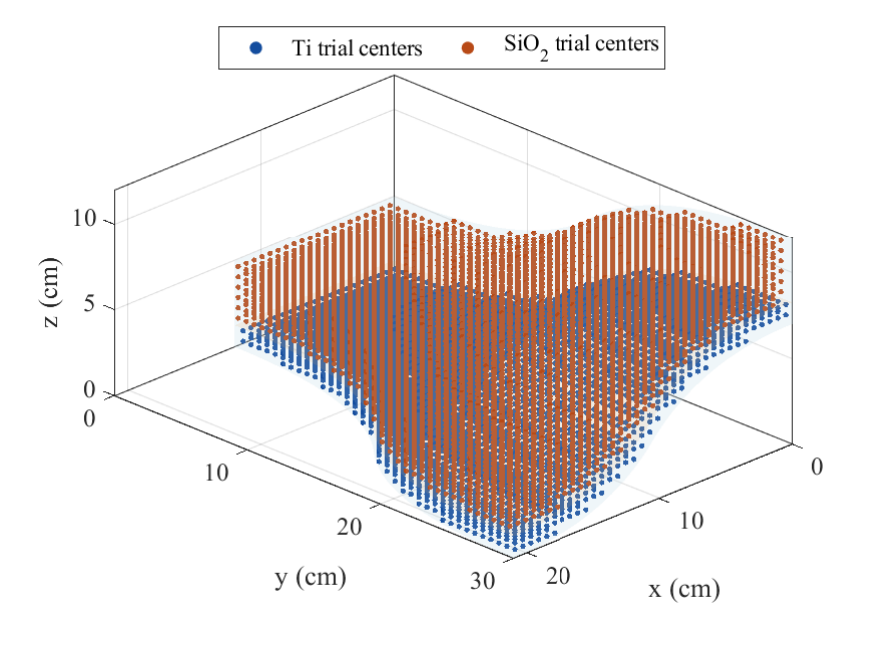}
    \caption{Independent subdomain trial centers.}
\end{subfigure}
\hfill
\begin{subfigure}[t]{0.32\textwidth}
    \centering
    \includegraphics[width=\textwidth]
    {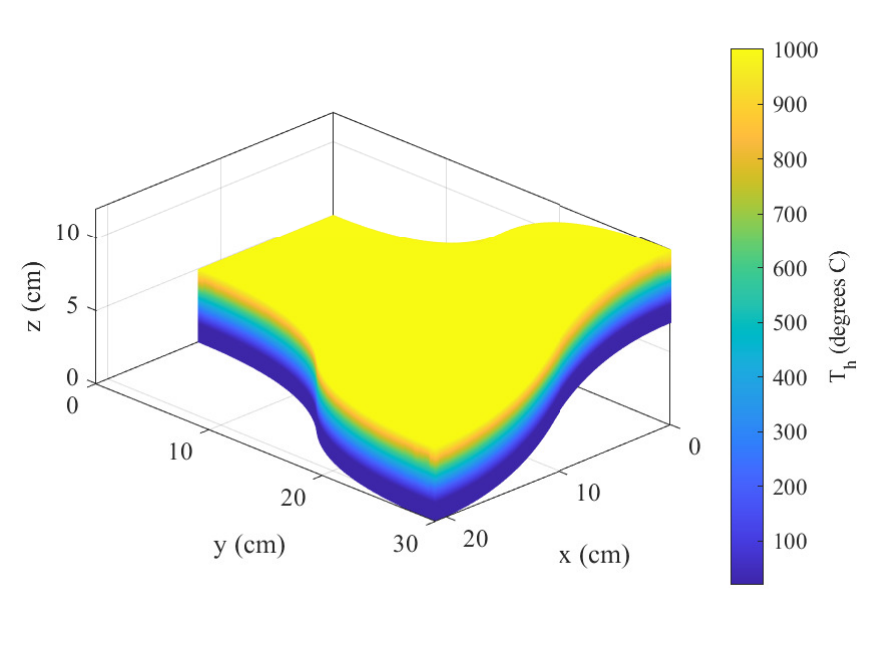}
    \caption{Numerical temperature field $T_h$.}
\end{subfigure}
\hfill
\begin{subfigure}[t]{0.32\textwidth}
    \centering
    \includegraphics[width=\textwidth]    {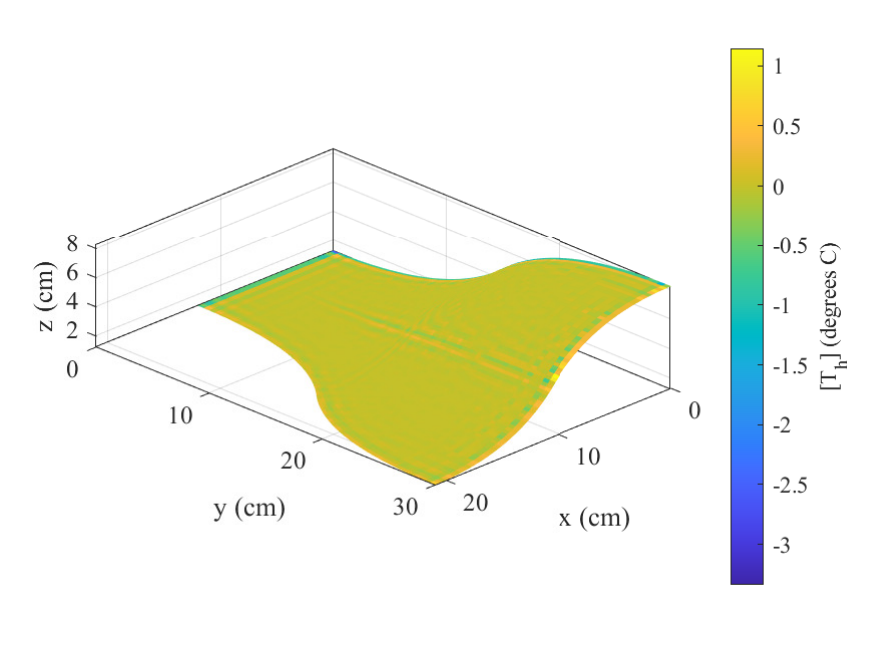}
    \caption{Numerical interface jump $\jump{T_h}$.}
\end{subfigure}
\vspace{-4mm}
\caption{
(Ex4)
Three-dimensional titanium--silica thermal-shielding tile
computed with a Mat\'ern kernel of Sobolev order $m=3$ and
$\varepsilon=0.3\sqrt2\,\mathrm{cm}^{-1}$.
The discretization uses $N=11320$ trial centers at
$h=0.6\,\mathrm{cm}$.}
\label{fig:example4-3d}
\vspace{-4mm}
\end{figure}

\begin{figure}[t]
\centering
\begin{subfigure}[t]{0.4\textwidth}
    \centering
    \includegraphics[width=\textwidth]
    {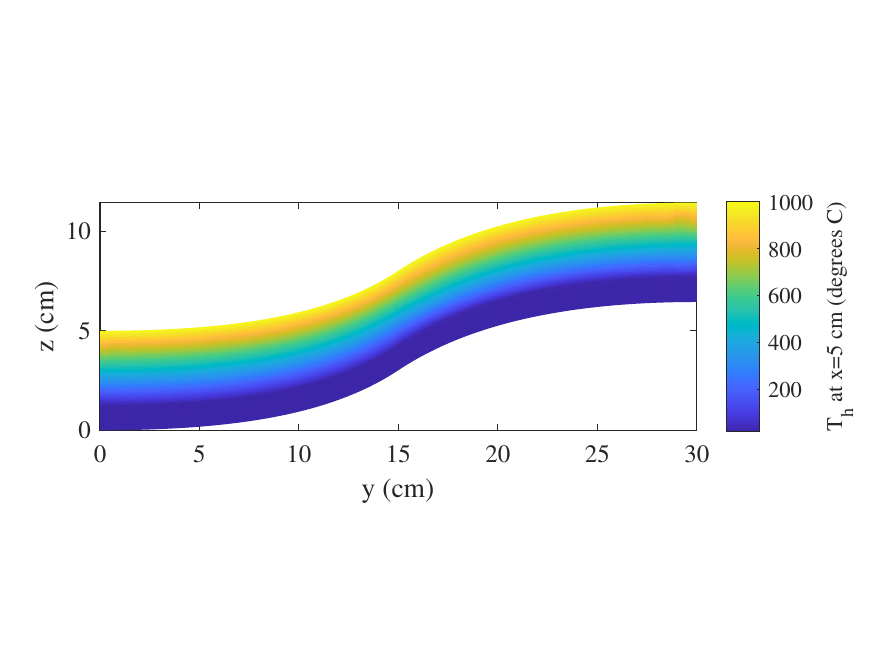}
    \caption{Temperature in the section $x=5~{\rm cm}$.}
\end{subfigure}
\hspace{6mm}
\begin{subfigure}[t]{0.36\textwidth}
    \centering
    \includegraphics[width=\textwidth]
    {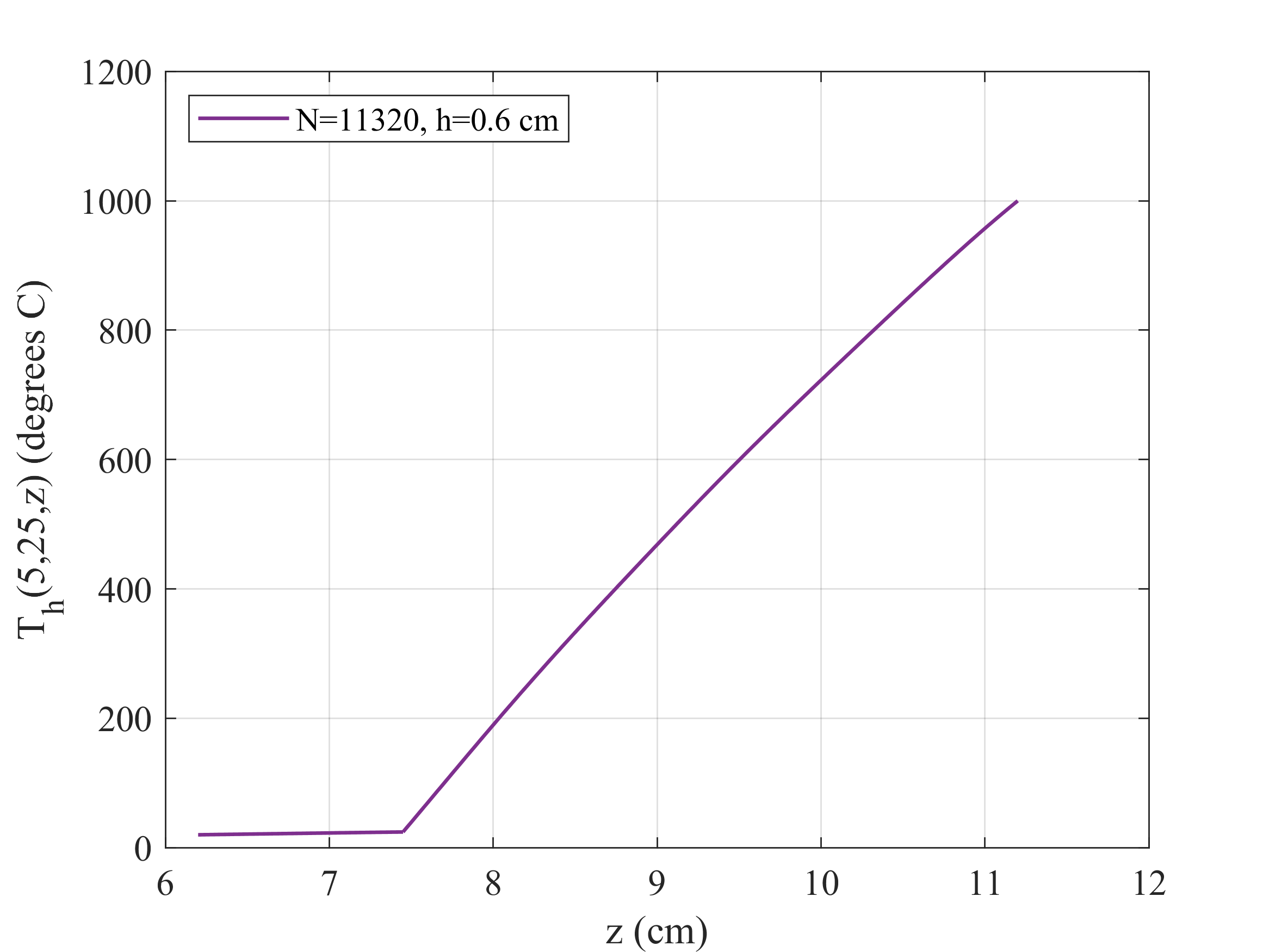}
    \caption{Through-thickness profile $T_h(5,25,z)$.}
\end{subfigure}
\vspace{-4mm}
\caption{
(Ex4)
Sectional and through-thickness temperature distributions
in the titanium--silica tile, computed with $N=11320$ centers
at $h=0.6\,\mathrm{cm}$ using the Mat\'ern kernel with
$m=3$ and $\varepsilon=0.3\sqrt2\,\mathrm{cm}^{-1}$.
}
\label{fig:example4-sections}
\vspace{-4mm}
\end{figure}

\section{Conclusions}
\label{sec:conclusions}

We established a trace inverse estimate for independent Sobolev-kernel
spaces on the two subdomains. The estimate controls the conormal
terms by the Nitsche energy, including configurations with an
interior inclusion, and yields well-posedness with penalty thresholds
independent of refinement. Together with approximation in the
Nitsche norms, it gives energy and $L^2$ error estimates under
the stated primal and adjoint regularity assumptions.
We also obtained a scale-explicit upper bound for the stiffness-matrix
condition number in the kernel-translate basis for
$0<\varepsilon\le1$.
The numerical experiments examine fixed-scale error reduction,
and the accuracy cost of
reducing kernel support. The three-dimensional example illustrates
the mixed-boundary implementation.
Extensions to evolving interfaces and to scale-dependent kernel
refinement, for which uniform stability and approximation estimates would
be required, remain subjects for future work.

\bibliographystyle{unsrt}
\bibliography{references}
\end{document}